\documentclass[11pt,regno]{amsart}
\usepackage[normalem]{ulem}
\usepackage{euscript,epstopdf,amscd,amsgen,times}
\usepackage{amssymb,amsmath,amsthm,amssymb,latexsym,mathtools}
\usepackage{color}
\usepackage[dvipsnames]{xcolor}
\usepackage{graphicx}
\usepackage{mathrsfs}
\usepackage[abs]{overpic}
\usepackage{amsfonts}
\usepackage{amssymb}
\usepackage{color}

\makeatletter
\newcommand*\bigcdot{\mathpalette\bigcdot@{.5}}
\newcommand*\bigcdot@[2]{\mathbin{\vcenter{\hbox{\scalebox{#2}{$\m@th#1\bullet$}}}}}
\makeatother

\newtheorem{mainthm}{Theorem}
\newtheorem{mcorol}[mainthm]{Corollary}

\newtheorem{thm}{Theorem}[section]
\newtheorem{mquestion}{Question}
\newtheorem{lem}[thm]{Lemma}

\newtheorem{prop}[thm]{Proposition}
\newtheorem*{thm*}{Theorem}

\newtheorem{defi}[thm]{Definition}

\theoremstyle{definition}

\newtheorem{claim}[thm]{Claim}
\newtheorem{remark}[thm]{Remark}

\newtheorem{terminology}[thm]{Terminology}

 \DeclareMathOperator{\diam}{diam}
 \DeclareMathOperator{\vol}{vol}
\newcommand{\eqdef}{\stackrel{\scriptscriptstyle\rm def}{=}}

\numberwithin{equation}{section}

\def\cS{\mathcal{S}}
\def\ffT{\mathfrak{T}}
\def\ffR{\mathfrak{R}}
\def\ffL{\mathfrak{L}}
\def\ffS{\mathfrak{F}}
\def\ffF{\mathfrak{F}}
\def\bft{{\mathbf{t}}}
\def\loc{{{\mathrm{loc}}}}

\makeatletter
\newcases{mycases}{\quad}{%
  \hfil$\m@th\displaystyle{##}$}{$\m@th\displaystyle{##}$\hfil}{\lbrace}{.}
\makeatother

\title[Homoclinic tangencies leading to robust heterodimensional cycles]{Homoclinic tangencies leading to robust heterodimensional cycles}
\author[Barrientos, D\'iaz, and P\'erez ]{Pablo G.~Barrientos, Lorenzo J. D\'\i az, and Sebasti\'an A. P\'erez}

\address{Institute of Mathematics and Statistics, University Federal Fluminense-UFF, Gragoata Campus, Rua Prof. Marcos Waldemar de Freitas Reis, S/n-Sao Domingos, Niteroi - RJ, 24210-201, Brazil}
\email{pgbarrientos@id.uff.br}

\address{Departamento de Matem\'atica PUC-Rio, Marqu\^es de S\~ao Vicente 225, G\'avea, Rio de Janeiro 225453-900, Brazil}
\email{lodiaz@mat.puc-rio.br}

\address{Instituto de Matem\'aticas, Pontificia Universidad Cat\'olica de Valpara\'\i so, Blanco Viel 596, Cerro Bar\'on, Valpara\'iso, Chile.}
\email{sebastian.perez.o@pucv.cl}

\begin{document}
\begin{abstract}
We consider  $C^r$  ($r\geqslant 1$)  diffeomorphisms $f$ defined on manifolds of
dimension $\geqslant 3$
with homoclinic
tangencies associated to saddles. Under generic properties,
we show that if
the saddle is homoclinically related to a blender
then the diffeomorphism $f$
can be {$C^r$} approximated by diffeomorphisms with
{$C^1$} robust heterodimensional cycles.
As an application, we show that
the classic Simon-Asaoka's examples of diffeomorphisms
with $C^1$ robust homoclinic tangencies also display
{$C^1$}  robust heterodimensional cycles.
In a second application, we consider homoclinic tangencies associated to
hyperbolic sets. When the entropy of these sets is large enough we obtain $C^1$ robust cycles
after $C^1$ perturbations.
\end{abstract}

\thanks{PGB and LJD were supported  [in part] by CAPES finance code 001 (Brazill), CNPq Projeto Universal and CNPq grants (Brazil). PGB was partially supported by grant PID2020-113052GB-I00 funded by MCIN/AEI/10.13039/501100011033 (Spain).
 LJD was partially supported by grant E-16/2014 INCT/FAPERJ  (Brazil).
SP was supported by Programa Postdoctorado FONDECYT 3190174 (Chile).}


\maketitle

\section{Introduction}
\label{s.introduction}

Homoclinic tangencies and heterodimensional cycles
associated to saddles are considered the main
bifurcation mechanisms generating nonhyperbolic dynamics. Palis conjecture \cite{Pal:00} claims that
 nonhyperbolic diffeomorphisms can be approached by systems displaying these bifurcations.
 The goal of this paper is to study the interplay between these two types of bifurcations.
A {\em{saddle}} (i.e., a hyperbolic periodic point with nontrivial stable and unstable bundles)
 of  a diffeomorphism has a {\em{homoclinic tangency}} if its stable and unstable manifolds
 have some nontransverse intersection.
 Two saddles with different {\emph{$u$-indices}} (i.e., dimension of the unstable bundle) have a
 {\em{heterodimensional cycle}} if their stable and unstable manifolds intersect  cyclically (by dimension deficiency, one of these intersections is
necessarily nontransverse). Homoclinic tangencies may occur in
manifolds of dimension two or higher, while the occurrence of
heterodimensional cycles requires dimension at least three.
 In what follows,  we will use the
term {\em{cycle}} to refer either to a homoclinic tangency or to a
heterodimensional cycle.

Note that cycles associated to saddles are breakable by small perturbations:
by Kupka-Smale theorem, periodic points of generic  diffeomorphisms are hyperbolic and their invariant manifolds are in general position (either they intersect transversely or are disjoint).
Even more, cycles may be  unbreakable if they are
associated to {\em{nontrivial}}  hyperbolic basic sets (i.e., not just
periodic orbits). This leads to the notions of \emph{$C^r$ persistent} and \emph{$C^r$ robust cycles.}
In the persistent case,  there is a $C^r$ open set where the diffeomorphisms
with cycles associated to saddles are $C^r$ dense. In the robust case, there is a hyperbolic set
with a cycle that cannot be destroyed by small  $C^r$ perturbations, see Definition~\ref{d.robustcycle}\footnote{Any robust cycle provides a persistent cycle.  As far as we know, it is unknown the existence of persistent cycles that are not robust (meaning that the saddle belongs to a hyperbolic set with a robust cycle).}.
Bearing in mind the importance of robust
dynamical properties, Bonatti proposed in~\cite{Bon:11} a
stronger version of Palis conjecture: every
nonhyperbolic diffeomorphism can be approached by diffeomorphisms
with a robust cycle.

Palis conjecture holds for
$C^1$ surface diffeomorphism \cite{PujSam:00} and there is some progress in higher dimensions,  see
\cite{CroPuj:15,CroSamYan:15}. Note that there are no surface diffeomorphisms
with $C^1$ robust homoclinic tangencies associated to basic sets, see
 \cite{Mor:11}.  In higher dimensions and  in the $C^1$ topology, Bonatti conjecture is true for the so-called
{\emph{tame diffeomorphisms,}} \cite{BonDia:08}.
For both conjectures, the $C^r$ case, $r \geqslant 2$, is widely open. The first reason for that is
that they
rely on the
  density of periodic points in the limit set, which is only known in the
$C^1$ case, see \cite{Pug:67,Pug:75}.  A second reason is the lack of connecting-like lemmas
allowing intersections between the invariant manifolds of saddles in the presence of  recurrences, see  \cite{Hay:97,BonCro:04}.
In view of these obstacles,
 the following question
(that can be posed in any regularity) provides
a bridge between these conjectures and a step towards  the one by Bonatti:
\begin{mquestion}\label{q.firstq}
When does a cycle associated to saddles yield (by small) perturbations robust cycles?
\end{mquestion}

The answer to this question depends on the type of cycle, the
dimension of the ambience, and the required regularity.
A strategy to find robust cycles is to see that a cycle can become robust
 when the saddle in the cycle belongs  to (or when the unfolding of the cycle generates) some {\em{massive}} horseshoe.
 This vague term means {\em{thick horseshoes}}
 or {\em{blenders,}} according to the type of cycle.
  We observe that $C^1$ diffeomorphisms cannot display robustly thick horseshoes, \cite{Ure:95},  while blenders may occur in dimension three or higher
in a $C^1$ robust way. This illustrates  some disparities between
the different  settings. Blenders will play a key role in this paper and will be discussed in
Section~\ref{s.blenders}. For the discussion in this introduction,  keep in mind that
a blender is a special type of hyperbolic set with some $C^1$ robust geometrical properties
guaranteeing intersections between the local unstable manifold of the blender
and a family of ``strong stable"   disks.

There are partial answers to Question~\ref{q.firstq}.
When $r \geqslant
2$, a homoclinic tangency of a $C^r$ diffeomorphism yields robust tangencies, see
\cite{New:79} in dimension two and
\cite{GonTurShi:93,PalVia:94,Rom:95} in higher dimensions. The generation of
robust tangencies in the $C^1$ setting in higher dimensions
depends crucially on the geometry of the tangency, see the discussion below
below. Finally, the question {\emph{does a  heterodimensional
cycle yield robust heterodimensional cycles?}} has a positive answer in the {\em{coindex}}
(i.e., the absolute value of the difference of the $u$-indices of the hyperbolic sets in the cycle) one case,
see
\cite{BonDia:08} for the $C^1$  case and
\cite{LiTur:xx} for the  $C^r$ case, $r\geqslant 2$.
The case of heterodimensional cycles of higher
coindex is widely open. We observe that all robust cycles generated in this paragraph are of the same type as the ones in the initial cycle.

An  issue related to Question~\ref{q.firstq}  is to determine the
interplay between heterodimensional cycles and homoclinic
tangencies. This leads to the following refinement of Question~\ref{q.firstq}:

\begin{mquestion}
\label{q.refinement}
 {\bf{(a)}} When does a homoclinic tangency associated to a saddle lead to robust
heterodimensional cycles? {\bf{(b)}} When does a
heterodimensional cycle associated to a pair of saddles lead to robust homoclinic tangencies?
\end{mquestion}

{The answers to these questions involves global dynamical aspects. In this paper, we focus on item (a), providing settings where homoclinic tangencies associated to saddles lead
to $C^1$ robust cycles after small $C^r$ perturbations\footnote{Note that if
$f\in \mathrm{Diff}^r(M)$, $r\geqslant 1$, then every $C^1$ neighbourhood of $f$ contains a
$C^r$ neighbourhood of $f$. Hence, if  $f$ has a $C^1$ robust cycle then it also has
a $C^r$ robust cycle.},
$r \geqslant 1$. Our constructions are motivated by the ones in \cite{LiTur:20}, where heterodimensional cycles are obtained
from  homoclinic tangencies with some symmetry conditions.
To describe our setting, we first comment some global aspects and restrictions concerning Question~\ref{q.refinement}.

{The presence of a homoclinic tangency implies
the existence of an invariant space of dimension at least two that cannot be decomposed into
one-dimensional invariant directions in a dominated way (see the appendix in Section~\ref{s.Sinovac} for the precise definitions). Thus partially hyperbolic diffeomorphisms  having
 dominated splittings  with one-dimensional bundles may display heterodimensional cycles but
 cannot exhibit homoclinic tangencies. This shows an obvious restriction to item (b).}
{When considering homoclinic tangencies a key aspect is how they are
embedded in the global dynamics. The following discussion illustrates this
aspect and motivates our constructions and hypotheses.}

To simplify the  analysis, let us assume that the ambiance has dimension three
and  the saddle with the homoclinic tangency has $u$-index one.
Under appropriate assumptions on the multipliers of the saddle,
the unfolding of this tangency yields
new hyperbolic periodic points with $u$-indices different from
one.
 Once saddles of $u$-index two are generated and the resulting dynamics exhibits saddles
 of different $u$-indices, it is necessary  to get cyclic heteroclinic relations between them.

 The occurrence of such heteroclinic relations is an intricate question that involves geometrical constraints.
 For instance,
 the tangency may occur ``inside"  a normally contracting hyperbolic surface and the saddle with the homoclinic tangency may have a Jacobian whose
 restriction to that surface is bigger than one.
 In this case, due to the persistence of normally hyperbolic surfaces,  the homoclinic tangency yields saddles contained in those surfaces with a normally contracting
 direction which are sources for the dynamics restricted to the surface. In this way, the tangency provides saddles of $u$-indices one and two. However, in the absence of further recurrences given by the global dynamics, there is no dynamical interaction between these saddles and heterodimensional cycles cannot occur. In rough terms, in the previous configuration,
 the dynamics can be ``reduced"  to a surface dynamics ``multiplied" by a contraction, we will refer to it   as ``essentially two-dimensional bifurcations".

We now   discuss the setting and the hypotheses of our results. Recall that hyperbolic basic sets of  surface diffeomorphisms cannot display $C^1$ robust
homoclinic tangencies. Thus to get $C^1$ robust homoclinic tangencies in higher dimensions
essentially two-dimensional dynamics must be avoided.
 To bypass this obstruction, in  \cite{BonDia:12}
 it is considered
a type of blender (called {\emph{blender-horseshoe,}} see Remark~\ref{r.blenderhorseshoe})
 with a homoclinic tangency.
Although
 \cite[Theorem 1.2]{BonDia:12}, see also \cite[Theorem 1 and Corollary 2]{BonCroDiaGou:13}, are  $C^1$ global results involving the recurrences provided by a {\em{homoclinic class}} (i.e., the closure of the set of transverse intersections between the invariant manifolds of a saddle)
they lead to the following heuristic principle:
a homoclinic class with robust {\emph{index variability}} (i.e., containing saddles of different
$u$-indices)
and a nondominated ``central" direction displays $C^1$ robust cycles
(simultaneously, homoclinic
tangencies and  heterodimensional cycles).  In the $C^1$ setting in
\cite{BonDia:12, BonCroDiaGou:13}
the index variability provides blenders inside the homoclinic  class as well as the following configuration:
a saddle
(satisfying some dissipative conditions)
with a homoclinic tangency and
homoclinically related to a blender.
The tangency provides the nondomination condition and also (after its unfolding) the index variability. The blender leads to dynamics that are not essentially two-dimensional.

This sort of conditions motivated the hypotheses in
Theorems~\ref{thm:main} and ~\ref{thm:main2} claiming
the existence of $C^1$  robust heterodimensional cycles arising from homoclinic tangencies after
small $C^r$ perturbations, $r \geqslant 1$.
Theorems~\ref{thm:main} considers  the  codimension one case (the saddle with the homoclinic tangency has $u$-index one)  while Theorems~\ref{thm:main2}  deals with the general setting
(any $u$-index). The precise statements can be found in Section~\ref{ss.homoclinictangency}.
The main difference between these two theorems is that in the codimension one case
the homoclinic tangency provides only saddles of $u$-indices one and two. As a consequence, only heterodimensional cycles associated to these saddles can appear and hence these cycles have coindex one. In the general setting,
the homoclinic tangency yields saddles of a wider range of $u$-indices, hence cycles of higher coindex  may appear.

We also provide settings to which our results apply.
For that, let us recall the first examples of $C^1$ robust homoclinic tangencies  given
by Simon  \cite{Sim:72b} and Asaoka
 \cite{Asa:08} built on the examples
 of nonhyperbolic diffeomorphisms in
 \cite{AbrSma:68}.
The key ingredient  in these constructions is the existence of a  hyperbolic set
 that is the local product of a nontrivial hyperbolic repeller
 (a DA-diffeomorphism in \cite{Sim:72b} and a Plykin repeller in \cite{Asa:08})
  and a contraction. We
 call this sort of hyperbolic set {\em{blenders of DA or Plykin type.}} A crucial property
 of this type of blender is that they contain
   their one-dimensional unstable manifolds.
 This property allows to get $C^1$ robust heterodimensional cycles in \cite{AbrSma:68,Sim:72a}
 (although this terminology
 is not used there). The idea in  \cite{Sim:72b,Asa:08}
 (reviewed in Section~\ref{s.simonasaokatan})
 is to produce homoclinic tangencies associated to
 these blenders.
In Corollary~\ref{c.p.asaoka-simon} (see also Remark~\ref{rem.asaoka-simon})
 we prove that these examples satisfy the hypotheses of  Theorem~\ref{thm:main}.  Consequently, for any $r\geqslant 1$, they lead to the simultaneous occurrence of $C^1$ robust
 homoclinic tangencies and heterodimensional cycles.
  Corollary~\ref{c.p.asaoka-simon}
 answers affirmatively a question posed in
\cite[Section 1.3]{BonDia:12}
in a discussion about
the $C^1$ coexistence of robust cycles and tangencies in  the examples in \cite{Asa:08}.

Finally, we exhibit settings where entropy-like assumptions imply the existence of
blenders. For that, we revisit the
$C^1$ constructions in \cite{AviCroWil:21}, showing that diffeomorphisms
with hyperbolic measures whose entropy is sufficiently large
exhibit blenders.  These results are summarised in
Theorem~\ref{mainthm-hyperbolic-measure}. Corollaries~\ref{cor:2}
and \ref{cor:3} translate our main results to that setting,
providing simultaneous $C^1$ robust heterodimensional cycles of a
variety of coindices.

We now state precisely our results.

\subsection{Robust heterodimensional arising from homoclinic tangencies}
\label{ss.homoclinictangency}

In the sequel, $\mathrm{Diff}^r(M^{n})$ denotes the space of $C^r$ diffeomorphisms of
a compact boundaryless manifold $M^n$ of dimension $n$ endowed with the uniform
$C^r$ topology. When we do not want to emphasise the dimension of the ambiance we will omit the superscript $n$.

The first ingredient in our results is the $cs$-blender.
Blenders are hyperbolic sets with an additional geometric structure guaranteeing
intersections between a family of ``stable"  disks and the local unstable manifold of the blender.
Blenders are also $C^1$ robust sets.  See Definition~\ref{def:blender}  for details.

Recall that two saddles of $f \in \mathrm{Diff}^r(M)$
 are {\emph{homoclinically related}} if the invariant manifolds of their orbits intersect transversely and cyclically. To be homoclinic related defines an
equivalence relation on the set of saddles of $f$. Two saddles that are homoclinically related have the same
$u$-index. The {\emph{homoclinic class}} of a saddle $P$ of $f$ is the closure of
the saddles of $f$ that are homoclinically related to $f$. A homoclinic class is a
{\emph{transitive set}} (i.e., it contains a point whose orbit is dense in the set).
As  a blender
$\Gamma$ of $f \in \mathrm{Diff}^r(M)$
is a hyperbolic basic  set (see  Definition~\ref{def:blender}), its saddles form a dense
subset of it and every pair of saddles of $\Gamma$ are
homoclinically related.  A
saddle is {\em{homoclinically related to the blender}}  if it is
homoclinically related to some saddle of the blender,
see (3) in Remark~\ref{r.blender}.

Recall that every hyperbolic set $\Lambda$ of a diffeomorphism $f$
has a well defined {\emph{continuation}}  for every
$C^1$ nearby diffeomorphism $g$.
We denote this continuation by  $\Lambda_g$.
If a hyperbolic set $\Lambda$ is transitive  then
the dimension of its unstable bundle  does not depend on the point
of $\Lambda$. This dimension is called the {\em{$u$-index}}
of $\Lambda$.
\begin{defi}[Robust cycles]\label{d.robustcycle}
{\em{Two transitive
hyperbolic sets $\Lambda$ and $\Upsilon$ of $f \in \mathrm{Diff}^r(M)$ with different
$u$-indices
 form
 a {\emph{heterodimensional cycle}} if their invariant stable and unstable
 sets intersect cyclically, that is,
 $$
 W^s( \Lambda,f)\cap W^u( \Upsilon,f)\ne\emptyset \quad \mbox{and} \quad
 W^u( \Lambda,f)\cap W^s( \Upsilon,f)\ne\emptyset.
 $$
The {\em{coindex}} of this cycle is the absolute value of the
difference of the $u$-indices of $\Lambda$ and $\Upsilon$.
This cycle is {\em{$C^r$ robust}} if there is a $C^r$ neighbourhood $\mathscr{N}$  of
$f$ such that the continuations  $\Lambda_g$ and $\Upsilon_g$ have a heterodimensional cycle
for every $g$ in $\mathscr{N}$.

A transitive hyperbolic set $\Lambda$ has a {\em{homoclinic tangency}}
 it its stable and unstable sets
have some nontransverse intersection. This homoclinic tangency is {\emph{$C^r$ robust}}
if there is a $C^r$ neighbourhood of
$f$ such that the continuation  $\Lambda_g$ of $\Lambda$ has a tangency
for every $g$ in that neighbourhood.}}
\end{defi}

We now describe the class of homoclinic tangencies studied in this paper.
Consider $f\in \mathrm{Diff}^r(M^{m+n})$, with $r, m, n \geqslant 1$ and $m+n\geqslant 3$,
 and
a saddle $P$ of $f$ of period $\pi$ and $u$-index $n$. Let
$\lambda_1,\dots,\lambda_m,\gamma_1,\dots,\gamma_n$ be the
eigenvalues of $Df^\pi(P)$ counted with multiplicity and ordered so that
\begin{equation}
\label{e.igualalambda}
 |\lambda_m|\leqslant \dots\leqslant |\lambda_1|= \lambda
 <1< \gamma=|\gamma_1|\leqslant \dots
\leqslant |\gamma_n|.
\end{equation}
The
multipliers of $P$ with modulus equal to $\lambda$ and $\gamma$
are called \emph{stable  and unstable leading
multipliers,} respectively. We denote by $m_s=m_s(P)$ and $n_u=n_u(P)$ the number
of stable and unstable leading multipliers, respectively, and
 say that $P$ is of
\emph{type $(m_s,n_u)$}.
The saddle $P$ is \textit{simple} if one of the following cases holds:
\begin{itemize}
\item
$P$ is of type $(1,1)$;
\item
 $P$ is of type
$(2,1)$ and  $\lambda_{1}$ and $\lambda_2$ are nonreal (and hence conjugate);
\item
 $P$ is of type
$(1,2)$ and  $\gamma_{1}$ and $\gamma_2$ are nonreal (hence conjugate);
\item
 $P$ is of type $(2,2)$,
 $\lambda_{1}$ and $\lambda_2$ are nonreal,
 and
  $\gamma_{1}$ and $\gamma_2$ are nonreal.
\end{itemize}

 We define the {\emph{leading Jacobian}}  $J_P(f)$ of $P$ as the product of the moduli of all leading multipliers of $P$ counted with multiplicity,
\begin{equation}\label{e.J}
J_P(f)\eqdef \lambda^{m_s}\,\gamma^{n_u}.
\end{equation}

We are now ready to state our main results.

\begin{mainthm} \label{thm:main}
 Consider  $f\in \mathrm{Diff}^r(M^{m+1})$, $r \geqslant 1,\, m\geqslant 2$,
 with a homoclinic tangency
associated to a  simple saddle $P$ of $u$-index one such that:
\begin{enumerate}
\item  $P$ is of type $(m_s,1)$ with $m_s\in \{1,2\}$,
\item $J_P(f)\geqslant 1$ and
\item
$P$ is homoclinically related to a $cs$-blender $\Gamma$ of
central dimension $m_s$.
\end{enumerate}
Then  there are diffeomorphisms $g$ arbitrarily $C^r$ close to $f$
 having simultaneously  $C^1$ robust
heterodimensional cycles  of coindex $i$ for every
$1\leqslant i  \leqslant m_s$. These cycles are associated to a hyperbolic basic set containing  $P_g$
and a hyperbolic basic set of $u$-index $1+i$.
\end{mainthm}

Let us briefly comment on the hypotheses of Theorem~\ref{thm:main}. Items
(1) and (3) are compatibility conditions between the dominated splittings of the saddle and
the blender.
Condition
(2) implies that the unfolding of the homoclinic tangency generates saddles of
$u$-index greater than one, enabling the creation of heterodimensional cycles between these new saddles and the continuations of the saddle in the initial cycle.

The next result extends the previous theorem when the saddles exhibiting the homoclinic tangency
have  $u$-index $n>1$.  A key ingredient in this result is the notion of a double blender, see
Definition~\ref{def:blender}.

\begin{mainthm} \label{thm:main2} Consider $f\in \mathrm{Diff}^r(M^{m+n})$,
$r\geqslant 1, m,n\geqslant 2$, with a homoclinic tangency
associated to a  simple saddle $P$ of $u$-index $n$
that is
\begin{enumerate}
\item  of type $(m_s,n_u)$ with $m_s, n_u \in \{1,2\}$
and
\item
 homoclinically related to a double blender of central
dimensions~$(m_s,n_u)$.
\end{enumerate}
Then there are diffeomorphisms $g$ arbitrarily $C^r$ close to $f$
having simultaneously $C^1$ robust heterodimensional cycles of
coindex $i$ for every $1\leqslant i \leqslant N$, where
$$
N=m_s, \quad \mbox{if} \quad J_P(f)\geqslant 1 \qquad  \mbox{and} \qquad
N=n_u, \quad \mbox{if} \quad  J_P(f)\leqslant 1.
$$
These cycles are associated to hyperbolic basic sets containing $P_g$ and
a hyperbolic basic set of $u$-index $n+i$.
\end{mainthm}

In Theorem~\ref{thm:main2}, we obtain robust cycles associated to the continuations of the blender and new saddles of different $u$-indices. It is an open question if there are robust cycles associated to hyperbolic sets containing these new saddles.

\subsection{Simultaneity of robust  homoclinic tangencies and  robust heterodimensional cycles}
\label{ss.simultaneous}
We now study the concurrence of robust cycles of different types.
We start with an observation.

\begin{remark}[Robust tangencies]\label{rem.Btangenciessim}
Blender-horseshoes are a particular type of $cs$-blenders, see Remark~\ref{r.blenderhorseshoe}.
Under the assumptions of Theorems~\ref{thm:main}   or \ref{thm:main2},
if the $cs$-blender $\Gamma$ of  $f$ is a blender-horseshoe
and
the saddle $P$ is of type $(1,1)$ then \cite[Theorem~4.9]{BonDia:12} implies
that there are diffeomorphisms $g$ arbitrarily $C^r$ close to $f$
 having  $C^1$ robust
tangencies associated to a hyperbolic basic set containing  $P_g$. This leads to $C^r$ open sets
whose closure contains $f$ consisting of diffeomorphisms having simultaneously $C^1$ robust heterodimensional cycles and $C^1$ robust homoclinic tangencies. We do not  know if
under the assumptions of Theorems~\ref{thm:main} or \ref{thm:main2}
and  with the general definition of a $cs$-blender (see Definition~\ref{def:blender})
such
a concurrence
holds.
\end{remark}

We observe that (in dimension three) there are several results assuring  the simultaneous  occurrence of the two types of robust cycles above.
The results in \cite{Li:17}   state conditions on the multipliers of a saddle-focus with a homoclinic tangency to generate simultaneously $C^1$ robust  homoclinic  tangencies and heterodimensional cycles by small $C^1$ perturbations. Similarly,  in~\cite{Bar:xx} it is proved that both types of robust cycles can be $C^1$ approximated
by diffeomorphisms with heterodimensional cycles associated to saddles with nonreal multipliers.
A  nondominated $C^r$ setting, $r\geqslant 2$, with simultaneous  occurrence
of  robust homoclinic tangencies and robust  heterodimensional cycles was explored in \cite{DiaPer:xx}.

We aim to get  the simultaneity of robust cycles in more general settings.
We apply
Theorem~\ref{thm:main} to the class of diffeomorphisms with robust
homoclinic tangencies
in~\cite{Sim:72b,Asa:08}. We see that these diffeomorphisms also display $C^1$ robust
heterodimensional cycles.
The key ingredient  in this result are the $cs$-blenders of
Plykin type obtained as  a ``local'' product of a Plykin repeller by
strong contractions, for details
see  Section~\ref{ss.blendersplykin}. These $cs$-blenders have fixed points
of $u$-index one. Roughly, the robust cycles in the next corollary  
are obtained bifurcating the homoclinic tangency of a fixed point in 
the blenders, see Figure~\ref{fig:asa-ok}.

\begin{mcorol} \label{c.p.asaoka-simon}
Consider
$f\in \mathrm{Diff}^r(M^{m+1})$, $r \geqslant 1,\, m\geqslant 2$,
with a  $cs$-blender $\Gamma$ of Plykin type whose  fixed point  $P$ has
leading multipliers $\lambda$ and $\sigma$
satisfying
$\lambda\, \sigma>1$. Suppose that
$f$ has homoclinic tangency associated to $P$.
Then there is a  $C^r$ open set  $\mathscr{C}$
whose closure contains $f$ such that every $g\in \mathscr{C}$
has simultaneously a  $C^1$ robust homoclinic tangency and
a $C^1$ robust
heterodimensional cycle associated to
the blender $\Gamma_g$.
\end{mcorol}

\begin{remark}\label{rem.asaoka-simon}
An ipsis litteris version of Corollary~\ref{c.p.asaoka-simon} can be stated for
 $cs$-blenders of DA type. In this case, the construction is analogous,  starting with
 a DA diffeomorphism defined on the two-dimensional torus, see
 for instance \cite[Chapter 7.8]{Rob:99}. This provides some constraints on the ambient manifold for the occurrence of  blenders of DA type. Note that Plykin attractors are defined locally in a ball, therefore there are no topological constrains for their occurrence.
\end{remark}

\begin{figure}
\centering
\begin{overpic}[scale=0.09,
]{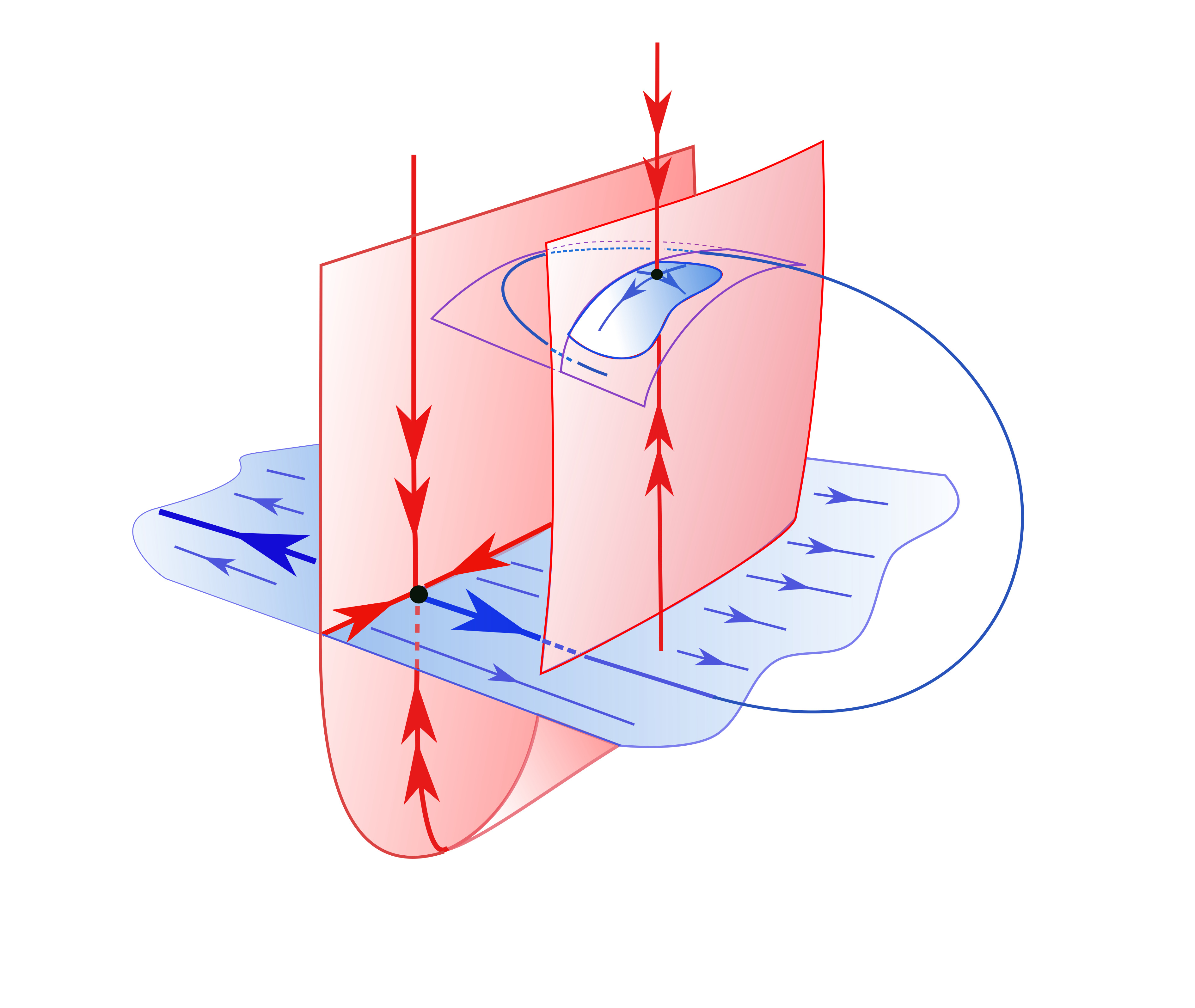}
  \put(97,127){\Large {$P_g$}}
      \put(20,160){\Large $W^u (\Gamma_g)$}
               \put(200,185){\Large $Q_g$}
 \end{overpic}
 \caption{\label{fig:asa-ok}Robust heterodimensional cycles in Corollary~\ref{c.p.asaoka-simon}}
\end{figure}

\subsection{$C^1$ robust heterodimensional cycles arising from horseshoes
with large entropy}
\label{ss.entropyblenders}

We now use Theorems~\ref{thm:main} and \ref{thm:main2}  to prove that
homoclinic tangencies associated to horseshoes with sufficiently large entropy lead
to $C^1$ robust heterodimensional cycles by small $C^1$ perturbations.
For this, we use the results in~\cite{AviCroWil:21} claiming that these horseshoes contain blenders.

We start with some preliminaries from ergodic theory.
In what follows, all measures considered are probability ones.
A measure $\mu$ of $f\in \mathrm{Diff}^r(M^{d})$ is called {\em{hyperbolic}} if it is
$f$-invariant, ergodic, and all its Lyapunov exponents are different from zero.
We list the Lyapunov exponents of $\mu$ (counted with multiplicity) as follows
\begin{equation} \label{e.Lypunov-spectrum}
\chi^{-}_{m}(\mu)\leqslant \dots \leqslant
\chi^{-}_1(\mu)<0<\chi_1^{+}(\mu)\leqslant \dots \leqslant
\chi^{+}_{n}(\mu), \quad m+n=d,
\end{equation}
and say that  $\mu$ has {\em{$u$-index}} $n$. To
emphasise the nature of these exponents, we write
\begin{equation} \label{e.Lypunov-cscu}
\chi^{cs}_\mu=\chi^{-}_1(\mu) \quad \mbox{and} \quad
\chi^{cu}_{\mu}=\chi^{+}_1(\mu).
\end{equation}
We also
consider the
{\em{ stable}} and {\em{unstable Jacobians of $\mu$}} defined by
\begin{equation} \label{e.Jacobians}
 \mathrm{J}^s_\mu \eqdef \mathrm{exp} \left( \sum_{i=1}^{m} \chi^{-}_i(\mu) \right) \quad
\text{and} \quad \mathrm{J}^u_\mu \eqdef   \mathrm{exp} \left(   \sum_{i=1}^{n}
\chi^{+}_i(\mu) \right).
\end{equation}

The previous discussion applies to any ergodic measure supported on a horseshoe, in particular to its measure of maximal entropy.  By a   \emph{horseshoe}
 of a
diffeomorphism we mean  a hyperbolic Cantor set that is
transitive and locally maximal.
In this case, the horseshoe $\Lambda$  has only one measure of maximal entropy
 ${\mu_{\Lambda,\max}} $. We denote by  $ h_{\mathrm{top}}(\Lambda,f)$ the topological entropy of $\Lambda$.

\begin{mcorol}
\label{cor:2} Consider $f\in \mathrm{Diff}^r(M^{m+1})$, $r > 1,\,
m\geqslant 2$,  having a homoclinic tangency
associated to a  simple saddle $P$ of $u$-index one such that
\begin{enumerate}
\item  $P$ is of type $(m_s,1)$ with $m_s\in \{1,2\}$,
\item  $J_P(f)\geqslant 1$, and
\item $P$  is homoclinically related to a horseshoe $\Lambda$
satisfying
$$
  h_{\mathrm{top}}(\Lambda,f) > - \log J^s_{\mu_{\Lambda, \max}} + \frac{1}{2r}\chi^{cs}_{\mu_{\Lambda, \max}}.
$$
\end{enumerate}
Then there are  diffeomorphisms $g$  arbitrarily $C^1$ close to $f$
displaying simultaneously $C^1$ robust heterodimensional cycles
of coindex $i$ for every $1\leqslant i  \leqslant m_s$. These
cycles are associated to $\Lambda_g$ and a transitive hyperbolic set of $u$-index $1+i$.
\end{mcorol}

Similarly, as a consequence of Theorem~\ref{thm:main2}, we obtain
the following.
\begin{mcorol}
\label{cor:3}
 Consider  $f\in \mathrm{Diff}^r(M^{m+n})$, $r > 1$ and  $m,\, n\geqslant 2$,
 having a homoclinic tangency
associated to a  simple saddle $P$ of $u$-index one such that
\begin{enumerate}
\item  $P$ is of type $(m_s,n_u)$ with $m_s, n_u\in \{1,2\}$,
\item $P$  is homoclinically related to a horseshoe $\Lambda$
with
$$h_{\mathrm{top}}(\Lambda,f) >
\max\{- \log J^s_{\mu_{\Lambda,\max}}
                                   + \frac{1}{2r}
                                   \chi_{\mu_{\Lambda, \max}}^{cs}, \ \log J^u_{\mu_{\Lambda, \max}}-
                                   \frac{1}{2r}\chi_{\mu_{\Lambda, \max}}^{cu}
                                   \}.
$$
\end{enumerate}
Then there are diffeomorphisms $g$ arbitrarily $C^1$ close to $f$
having simultaneously $C^1$ robust heterodimensional cycles of
coindex $i$ for every $1\leqslant i \leqslant N$, where
$$
N=m_s, \quad \mbox{if $J_P(f)\geqslant 1$} \qquad  \mbox{and} \qquad
N=n_u, \quad \mbox{if $J_P(f)\leqslant 1$.}
$$
These
cycles are associated to $\Lambda_g$ and a transitive hyperbolic set of $u$-index $n+i$.
\end{mcorol}

We now state  the variation
 of the results in \cite{AviCroWil:21} mentioned above (see  Theorem~\ref{mainthm-hyperbolic-measure}).
A first step of the proof of this variations involves linearisation and
diagonalisation results (see  \cite[Theorems B and 3.2
]{AviCroWil:21}). These results  claim that given any diffeomorphism $f$ with a
horseshoe $\Lambda$ there is a $C^1$ local perturbation $g$  of $f$ having an {\em{affine
horseshoe}} $\widetilde \Lambda$ close to $\Lambda$ in the
Hausdorff distance and whose topological
entropy is also close to the one of $\Lambda$. Moreover, the set $\widetilde
\Lambda$ has a dominated splitting consisting of one-dimensional
bundles. A second step, also borrowed from \cite{AviCroWil:21}, is
that hyperbolic measures with ``large entropy"  (in the
sense of Equation \eqref{e.largeentropy},  see
also Remark~\ref{r.Ruelle-inequality})  provide blenders, see
\cite[Theorems B' and C]{AviCroWil:21}.

Finally, recall that given any hyperbolic measure $\mu$ of a diffeomorphism $f$,
Pesin theory  provides stable and
unstable manifolds, denoted by $W_\mu^{\ast}(x,f)$, $\ast=s,u$,
for $\mu$-almost every point $x$. Note that this requires $f$ to be $C^r$ with $r>1$.
We say that a
transitive hyperbolic set $\Lambda$ and a hyperbolic measure $\mu$
are {\em{homoclinically related}} if there is a periodic orbit
$\mathcal{O}\subset \Lambda$  such that
$W^s(\mathcal{O},f)\pitchfork W^u_\mu(x,f)\not=\emptyset$ and
$W^u(\mathcal{O},f)\pitchfork W_\mu^s(x,f) \not = \emptyset$ for
$\mu$-almost every point $x\in \Lambda$.

\begin{thm} \label{mainthm-hyperbolic-measure}
Consider $f\in \mathrm{Diff}^r(M^{m+n})$, $r>1$, $m \geqslant 2,
n\geqslant 1$,
and a hyperbolic measure $\mu$ of  $f$ with
$u$-index $n$
such that
\begin{equation}
\label{e.largeentropy}
   h_\mu(f) >   -\log \mathrm{J}^s_\mu + \frac{1}{2r} \chi^{cs}_\mu.
\end{equation}
Then there is a sequence $(g_k)$, $g_k\in
\mathrm{Diff}^r(M^{m+n})$,
 of local perturbations  of $f$ supported on
arbitrarily small neighbourhoods of the support of $\mu$ and
converging to $f$ in the $C^1$ topology such that every $g_k$
has a $cs$-blender $\Gamma_k$ such that

\begin{enumerate}
\item\label{i.bl1}
$\Gamma_k$ has
central dimension $d_{cs}=m -1$
and
a dominated splitting consisting of  one-dimensional
subbunddles,
\item \label{i.bl2}
$\Gamma_k$ converges to the  support of  $\mu$ in Hausdorff
distance, and
\item \label{i.bl3}
$h_{\mathrm{top}}(\Gamma_k,g_k)$ converges to $h_{\mu}(f)$.
\end{enumerate}

Moreover, if $f$ has a saddle $P$ that is homoclinically related to the measure
$\mu$ then the sequence $(g_k)_k$ can be chosen
such that the continuation $P_{k}$ of $P$ for $g_k$ is
homoclinically related to the blender $\Gamma_{k}$. In addition, if $P$ has a homoclinic tangency,
then $g_k$ can be chosen with a homoclinic tangency associated to $P_k$.
\end{thm}

\begin{remark} According  to item (2)  in Remark~\ref{r.blender}, the
blenders $\Gamma_k$ in Theorem~\ref{mainthm-hyperbolic-measure}
are $cs$-blenders of central dimension $k$ for every
$k=1,\dots,m-1$ (here $m$ is the number of negative Lyapunov
exponents of $\mu$ counted with multiplicity).
\end{remark}

\begin{remark} \label{r.Ruelle-inequality}
A result analogous to Theorem~\ref{mainthm-hyperbolic-measure}
holds under the assumption
\begin{equation} \label{e.almost-Pesin}
  h_{\mu}(f) > \log \mathrm{J}^u_\mu - \frac{1}{2r}\chi^{cu}_\mu
\end{equation}
called in \cite{AviCroWil:21} {\em{almost Pesin formula.}} This
condition leads to $cu$-blenders with similar properties as  in Theorem~\ref{mainthm-hyperbolic-measure}.
In particular,
 if $\mu$ is a  hyperbolic measure
satisfying the  Pesin entropy formula (that is, $h_{\mu}(f)=\log
\mathrm{J}^u_\mu$) then there are arbitrarily small $C^1$
perturbations $g$ of $f$
 having $cu$-blenders $\Gamma_g$ satisfying conditions (1)--(3) in Theorem~\ref{mainthm-hyperbolic-measure}.
 Clearly, in item (1) the central dimension $d_{cu}$ of the blender
  is
$k=1,\dots, n-1$, where $n$ is the $u$-index of $\mu$.
\end{remark}

\begin{remark} \label{r.superblender} Double blenders
(see Definition~\ref{def:blender})
with
central dimensions $(m-1,n-1)$,
where $m$ and $n$ are as in the previous remarks,
 and dominating splittings
into one-dimensional subbundles are obtained when
\begin{equation*} \label{e.almost-Pesin-double}
  h_{\mu}(f) > \max\{ -\log \mathrm{J}^s_\mu + \frac{1}{2r}
  \chi^{cs}_\mu, \
 \log \mathrm{J}^u_\mu - \frac{1}{2r}\chi^{cu}_\mu\}.
\end{equation*}
\end{remark}

\begin{terminology}\label{notation1}
Throughout this paper, we will use the term \emph{perturbation} only to refer
to arbitrarily small ones.
\end{terminology}

\subsection*{Organisation}
This paper is organised as follows. In Section~\ref{s.blenders},  we introduce and discuss
$cs$ and $cu$-blenders. In Section~\ref{s.dominatedhorseshoed},
we study the generation of simple saddles with
some prescribed domination properties. In Section~\ref{s.dynamicsatthetangency}, we analise
the return maps in a neighbourhood of the homoclinic tangency and explore
the expanding-like properties derived from the expanding leading Jacobian hypothesis.
Section~\ref{s.indexvariation} introduces unfolding families and uses them to get new saddles
of different $u$-indices. Theorems~\ref{thm:main} and ~\ref{thm:main2}  are proved in Section~\ref{s.proofofmaintheorems}, for that we relate the saddles with different $u$-indices previously obtained with the blender.
 Section~\ref{s.simonasaokatan} is dedicated to Simon-Asaoka's examples.
Section~\ref{s.appendix}
is an appendix dedicated to the proof of Theorem~\ref{mainthm-hyperbolic-measure}.
Finally, Section~\ref{s.Sinovac} is another appendix about dominated splittings.

\section{Blenders} \label{s.blenders}
There are several definitions of
blenders, each one adapted to an appropriate
 context  and emphasising some
properties.  We do not aim to discuss and compare these definitions.
As mentioned above,  blenders are hyperbolic sets with a geometric structure guaranteeing
intersections between a family of ``stable"  disks and the local unstable manifold of the blender.
A fundamenta property of blenders is their $C^1$ robustness.
 A relevant point in our
Definition~\ref{def:blender} is the existence of a distinctive saddle. As far as we
know, distinctive saddles exist in any known example of blenders, their existence  is an
 explicit requirement in~\cite[Definition A1]{Ber:16}\footnote{ Let us summarise some settings where blenders play key roles.
 Blenders were initially defined  as having  central dimension one. Currently, there  are versions
 (as the ones here) where the central dimensions are greater than one.
Applications of  blenders of central dimension one include: robustly transitive
dynamics~\cite{BonDia:96}, robust heterodimensional cycles~\cite{BonDia:08,LiTur:xx},
robust homoclinic tangencies~\cite{BonDia:12}, stable ergodicity~\cite{RRTU}, and
construction of nonhyperbolic ergodic measures~\cite{BocBonDia:16}, among others. Blenders with larger central dimensions were introduced
in~\cite{NasPuj:12,BarRai:18} to study
 instability problems in symplectic dynamics, in~\cite{BarKiRai:14}
 to obtain
 robust heterodimensional cycles of large coindex,  and in~\cite{BarRai:17,Asa:xx} to
 get robust tangencies of large codimension.
Blenders of large central dimension also appear in the study of  ergodicity of conservative
partial hyperbolic systems \cite{AviCroWil:21},
of  holomorphic dynamics \cite{Taf:21,Duj:17,Bie:20}, and of parametric families of maps  (endomorphisms in \cite{PB},\cite{BCP} and diffeomorphisms~\cite{BarRai:21}).}.

Our definition of a  blender involves the
concepts of hyperbolicity and partial hyperbolicity, stable and
unstable bundles,  strong invariant manifolds, and domination. For
these concepts, see the appendix in Section~\ref{s.Sinovac}. In what follows, we
write dominated splittings $E_1  \oplus \cdots \oplus E_k$ with
several bundles in such a way the bundle $E_i$ is more ``contracting''
than the bundle $E_{i+1}$.

We will use the following notation for the local invariant manifolds
of a hyperbolic set $\Lambda$ of a diffeomorphism $f$ relative to a neighbourhood $U$
of $\Lambda$. We let
\begin{equation}
\label{e.localmanifolds}
  W^u_{U, \mathrm{loc}}(\Lambda,f) \eqdef \{ x\in U \colon f^i (x) \in U\, \mbox{for every $i\leqslant 0$}\}.
\end{equation}
The set
$  W^s_{U, \mathrm{loc}}(\Lambda,f)$ is defined in the obvious way. A special case occurs when $\Lambda$ is
a periodic orbit.

\begin{defi}[$cs$-, $cu$-blender, and double blender]
\label{def:blender}
{\em{Let $f\in \mathrm{Diff}^1(M^d)$, $d\geqslant 3$.
A hyperbolic and transitive set $\Gamma \subset {M}^d$ is a
\emph{$cs$-blender } if the following holds:}}

\smallskip

\noindent {\rm{(a) (local maximality)}}
{\em{There is an open neighbourhood $U$ of $\Gamma$ such that
$$
\Gamma = \bigcap_{i\in\mathbb{Z}} f^i (\overline{U}).
$$}}

\smallskip

\noindent
{\rm{(b) (partial hyperbolicity)}}
{\em{There is an $Df$-invariant dominated splitting with three nontrivial bundles
defined over $\Gamma$
$$
T_\Gamma M^d = E^{ss}\oplus E^c \oplus E^{u},
$$
where
 $E^s=E^{ss}\oplus E^c$ and
 $E^u$ are the stable and unstable bundles of $\Gamma$.
 We write
$$
d_{cs}\eqdef \dim E^{c}\geqslant 1
\quad  \mbox{and} \quad d_{ss}\eqdef \dim E^{ss}\geqslant 1.
$$}}

\smallskip

\noindent
 {\rm{(c) (open set of embedded disks)}}
{\em{There is an open set $\mathscr{D}^{ss}$ of $C^1$-embeddings of
$d_{ss}$-dimensional disks into $M^d$  such that:
 \begin{enumerate}
\item[i)] \label{c1}
  {\rm{(distinctive saddle)}}
There is a saddle $Q^\ast\in
\Gamma$ whose strong stable manifold $W^{ss}(Q^\ast,f)$
contains a disk of $\mathscr{D}^{{ss}}$ containing $Q^\ast$ in its interior.
\item[ii)] \label{c2}
{\rm{(robust intersections)}} There is a
 $C^1$-neighbour\-hood $\mathscr{U}$ of $f$ such that
for every $g\in \mathscr{U}$
the continuation $\Gamma_g$ of $\Gamma$ satisfies
$$
  W^u_{U, \mathrm{loc}}(\Gamma_g,g) \cap D \not = \emptyset,
\quad \mbox{for every} \quad D\in \mathscr{D}^{{ss}}.
$$
  \end{enumerate}

The set $U$ is called a {\emph{reference domain}} of the blender.
The sets $\mathscr{D}^{{ss}}$ and $$B^{ss}\eqdef
\mathrm{int}\Big(\bigcup_{D\in\mathscr{D}^{ss}}D\cap U\Big)$$ are
called {\emph{superposition region}}
 and {\emph{superposition
domain}} of the blender, respectively. The point $Q^\ast$ is {\emph{a
distinctive saddle}} of the blender. We say that $d_{cs}$ is the
{\em{central dimension}} of the blender.

A \emph{$cu$-blender} of $f$ is $cs$-blender for $f^{-1}$.  For a
$cu$-blender, $d_{cu}$, $\mathscr{D}^{{uu}}$, and $B^{uu}$ are
defined as above.

 A \emph{double blender of central dimensions
$(d_{cs},d_{cu})$} is a hyperbolic set that is simultaneously a
$cs$-blender of central dimension $d_{cs}$ and ${cu}$-blender of
central dimension $d_{cu}$.
In particular, a double blender has a pair of distinctive saddles $Q^\ast_{cs}$ and $Q^\ast_{cu}$
as well as $ss$- and $uu$-superposition
regions and domains. }}
  \end{defi}

Note that $cs$-blenders with central dimension $d_{cs}\geqslant 1$ can occur only in dimension
$1+d_{cs}+1 \geqslant 3$. Similarly, double blenders with central dimensions $(d_{cs}, d_{cu})$ can appear only in dimension
$1+d_{cs}+d_{cu}+1\geqslant 4$.

\begin{remark}[Blender-horseshoes]\label{r.blenderhorseshoe}
Blender-horseshoes were introduced in \cite{BonDia:12}.
They have a one-dimensional central dimension. Their key property is that
iterations of the disks in the superposition region provide new disks of that region.
Every blender-horseshoe is a $cs$- or a $cu$-blender with central one-dimensional direction.
\end{remark}

\begin{remark}[Properties of blenders]
\label{r.blender}
Consider $f\in \mathrm{Diff}^r(M)$, $r \geqslant 1$, with
a $cs$-blender $\Gamma$ of central dimension $d_{cs}$,
 reference domain $U$,
partially hyperbolic splitting
\begin{equation}
\label{e.dsblender}
T_\Gamma M^d = E^{ss}\oplus E^c \oplus E^{u}, \qquad  d_{cs}= \dim E^c,
\quad
d_{ss} = \dim E^{ss},
\end{equation}
superposition region $\mathscr{D}^{ss}$, and distinctive saddle $Q^\ast$.

\medskip

\noindent
(1) {\bf Continuations of
blenders:}
 Having a $cs$-blender is a $C^1$ open condition (and
hence $C^r$ open). For every $g$ sufficiently $C^1$
close to $f$ the continuation $\Gamma_g$ of $\Gamma$ is also a
$cs$-blender with the same reference domain  $U$, the same
superposition region $\mathscr{D}^{ss}$, and the same central
dimension. Moreover, if $Q^\ast$ is a distinctive saddle of
$\Gamma$ then its continuation $Q^\ast_g$ is also a distinctive saddle
of $\Gamma_g$.

\smallskip

\noindent
(2)  {\bf Central dimensions of a $cs$-blender:}
The hyperbolicity of  $\Gamma$ implies that if
$D$ is a disk in $\mathscr{D}^{ss}$
 and
$X\in
W^{u}_{U, \mathrm{loc}}(\Gamma,f)\cap D$ then
there is $Z\in\Gamma$ such
that $X\in W^{u}_{U, \mathrm{loc}}(Z,f)\cap D$. Moreover, the codimension of
$$
T_X W^{u}_{U, \mathrm{loc}} (Z,f)+T_XD
$$
is at least $d_{cs}$ and is  equal to $d_{cs}$ if and only if
the intersection between $W^{u}_{U, \mathrm{loc}}(Z,f)$ and $D$ at $X$ is
quasi-transverse (i.e., $T_X W^{u}_{U, \mathrm{loc}}(Z,f)\cap T_XD=\{\bar 0\}$).

Assume now that the bundle $E^c$ of  $\Gamma$ splits in a
dominated way as follows
\begin{equation}
\label{e.splittingcentral}
E^c = E^c_1 \oplus \cdots \oplus  E^c_k.
\end{equation}
For each $1\leqslant r\leqslant \ell \leqslant k$, consider the bundles and the dimensions
$$
E^c_{r,\ell} \eqdef E^c_r \oplus \cdots \oplus  E^c_{\ell},
\qquad
d_{r,\ell} \eqdef \dim E^c_{r,\ell} .
$$
We claim that
  $\Gamma$ is also a $cs$-blender of central dimension
  $d_{r,k}$ for each $r=2, \dots,k$. Using \eqref{e.splittingcentral},
for each $r \in \{2, \dots, k\}$, we get a dominated splittings of $\Gamma$ of the form
\begin{equation}
\label{e.newsplitting}
T_\Gamma M^d =
(E^{ss}\oplus E^c_{1,r-1} )
\oplus
E^c_{r, k}
 \oplus E^{u},
 \end{equation}
where $E^{ss}\oplus E^c_{1,r-1}$ is a strong stable bundle and
$E^c_{r, k}$ is a center stable bundle  of dimension $d_{r,k}$.

To see that  $\Gamma$ is a $cs$-blender of central dimension
  $d_{r,k}$ it is sufficient to consider the splitting in   \eqref{e.newsplitting}
 and any family of disks $\widehat{\mathscr{D}}^{ss}$ of dimension $d_{ss}+ d_{1,r-1}$
 containing the initial family $\mathscr{D}^{ss}$ and such that
 some disk of this family contains $Q^\ast$ in its interior and is contained in
 the strong stable manifold $W^{ss}(Q^\ast)$ (now considered with respect the
strong stable bundle $E^{ss}\oplus E^c_{1,r-1}$).
  By the observations above, the intersections properties in
 (c) hold.

 In the special case when $E^c$ splits into one-dimensional bundles in a dominated way, we have that
 $\Gamma$ is a $cs$-blender of central dimensions $1, \dots, d_{cs}$.
 This sort of blender was introduced in \cite{NasPuj:12,BarKiRai:14} and called  \emph{superblender} in \cite{AviCroWil:21}.

\smallskip

\noindent
(3)  {\bf Homoclinic relations:}
Note that every pair of saddles of $\Gamma$ are
homoclinically related. In particular, every saddle of $\Gamma$ is
homoclinically related to any distinctive saddle of $\Gamma$. Thus any
saddle homoclinically related to the blender  is
homoclinically related to
any distinctive saddle of the blender.
\end{remark}

The following remark is a standard consequence of the inclination lemma and the comments above.

\begin{remark}[Tangencies associated to distinctive saddles]\label{r.tangenciesandcontinuations}
Let $f\in \mathrm{Diff}^r(M)$ and  consider two saddles  $R$ and $S$ of $f$ which are homoclinically related. Suppose that
$R$ has a homoclinic tangency. Then there is a $C^r$ perturbation $g$ of $f$ such that $S_g$ has a homoclinic tangency. We get the following consequence. Fix a $cs$-blender $\Gamma$ of $f$,
a distinctive saddle $Q^\ast$ of $\Gamma$, and
a saddle $P$ homoclinically related
to $\Gamma$.
Then if $P$ has a homoclinic tangency  there is a $C^r$  perturbation $g$ of $f$ having a homoclinic tangency associated to the distinctive saddle $Q^\ast_g$ of the blender $\Gamma_g$.

A variation of the above assertion is the following. Consider  a hyperbolic basic set $\Lambda$ of $f$ (i.e., locally maximal in
an open neighbourhood of it) with a homoclinic tangency (not necessarily associated to a saddle). Given any periodic saddle $R\in \Lambda$ there is a $C^r$ perturbation $g$ of $f$ with  a homoclinic tangency associated to  $R_g\in \Lambda_g$.
\end{remark}


\section{Homoclinic relations of simple saddles with prescribed type}
\label{s.dominatedhorseshoed}

This section aims to prove
Proposition~\ref{p.Q} about the generation of simple saddles with
some prescribed domination. The relevant cases lead to simple
saddles of type $(2,1)$ and $(2,2)$.

For the next definition, note that if a saddle $P$ of a diffeomorphism
$f$ has $u$-index one then $W^u_{\mathrm{loc}}(P,f)\setminus
\{P\}$ has two connected components\footnote{A local unstable manifold of $P$ is
any disk contained in $W^u (P,f)$ of the same dimension as $W^u(P,f)$ that
contains $P$ in its interior. A local stable manifold is similarly defined.}.
In this case, we say that
the saddle $P$ is \emph{biaccumulated by its transverse homoclinic
points} if $W^s(P,f)$ transversely intersects both connected
components of $W^u_{\mathrm{loc}}(P,f)\setminus\{P\}$.
Note that this definition does not depend on the choice of the local manifold.
We will use this  property
in Section~\ref{s.returnmapcoidmensionone}
 to define return
maps associated to a homoclinic tangency of a saddle with $u$-index one.

For the next proposition, recall that a saddle $P$ of $f$ has {\emph{simple spectrum}}
if the eigenvalues of $Df^\pi (P)$ are all real, positive,  and have multiplicity one, here $\pi$ is the period of $P$.

\begin{prop} \label{p.Q}
Consider $f\in \mathrm{Diff}^r(M^{m+n})$, $r\geqslant 1$,
$m\geqslant 2$, $n\geqslant 1$, with a pair of
saddles $P$ and $P'$ of $u$-index $n$  which are homoclinically related.
Assume that $P$ is simple of type $(m_s,n_u)$ and that $P'$ has simple spectrum.
Then given any neighbourhood $V$ of $P'$ there is a $C^r$
perturbation $g$ of $f$ having a saddle $Q \in V$ such
that
\begin{enumerate}
\item $Q$ is simple of type $(m_s,n_u)$ and
$\log J_P (f)\cdot \log J_{Q}(g)>0$,
\item  $Q$ is homoclinically related to the continuations $P_g$ and
$P'_g$.
\item
If $n=1$ then $Q$ is
biaccumulated by its transverse homoclinic points.
\end{enumerate}
Finally, if  $P$ has a homoclinic tangency then $Q$ can be taken
with a homoclinic tangency.
\end{prop}

For simplicity, in what follows, let us
assume that the saddles have period one.

We will use this  proposition  to get saddles
with nonreal central eigenvalues.
Thus the relevant case occurs when $m_s$ or $n_s$ (or both) are different from one.
 In the
literature  many results provide conditions for the
abundance of saddles with simple  spectrum
after perturbations. In particular, the case $m_s=n_u=1$ is well-known
(see, for instance, the linearisation result in \cite[Theorem 3.2]{AviCroWil:21}).
 Thus  we consider the cases when $m_s \, n_u >1$ and skip
 the case $(1,1)$.

Let us note that the condition on the existence of a saddle homoclinically related to
$P$ with simple spectrum is not a restriction in our context.
Indeed,  if a saddle $P$ is homoclinically related to a saddle $\widetilde P$,
 given any neighbourhood $V$ of $\widetilde P$
after
a $C^r$ perturbation of $f$ we can assume that there is a saddle $P'\in V$
with simple spectrum that is
homoclinically related to $P$ and  $\widetilde P$.

The key step in the proof of Proposition~\ref{p.Q} is the next lemma:

\begin{lem}\label{l.nonreal}
Consider $f\in \mathrm{Diff}^r(M^{m+n})$ and saddles $P$ and
$P'$ as in Proposition~\ref{p.Q}.  Assume
that $P$ is of type $(2,n_u)$  and $C^r$ linearisable in a
neighbourhood $U$. Given any
 neighbourhood $V$ of $P'$
 there is a  $C^r$
perturbation $g$ of $f$ with a hyperbolic transitive set  $\Lambda_g$
 containing $P_g, P_g',$ and a simple saddle $Q$ of type
 $(2,n_u)$ such that the
 orbit of $Q$ intersects $U$ and $V$ and satisfies
$$
 \log J_{P_g} (g) \cdot \log J_{Q} (g)>0.
 $$
\end{lem}

Before proving the lemma, let us  give a simple two-dimensional geometrical argument
leading to matrices with a pair of conjugate nonreal eigenvalues.

Consider a $2\times 2$ real matrix $A$ with
real eigenvalues $\tau>\rho >0$ and  the
family of linear maps
\begin{equation}
\label{e.rotation} A_\varphi = A \circ R_\varphi, \quad
\mbox{where} \quad
R_\varphi = \left( \begin{matrix} \cos \varphi &  -\sin \varphi\\
\sin \varphi & \cos \varphi
\end{matrix}
\right), \quad \varphi\geqslant 0.
\end{equation}
\begin{claim}[Saddle-node bifurcations]\label{c.r.complex} There is $\varphi_0>0$ such that
\begin{itemize}
\item
$A_{\varphi}$ has different and positive real eigenvalues for every $\varphi \in [0, \varphi_0)$,
\item
 $A_{\varphi_0}$
has a real eigenvalue of multiplicity two, and
\item
$A_{\varphi}$ has a pair of nonreal eigenvalues for
$\varphi>\varphi_0$ close enough to $\varphi_0$.
\end{itemize}
\end{claim}

\begin{proof}
Without loss of generality, we can assume that the eigenvectors  of $A$ are $v =(1,0)$
and $w=(\cos \theta, \sin \theta)$,
$\theta \in (0,\pi)$,
associated $\tau>\rho >0$.
Consider the circle map  $\widetilde{A}_\varphi$ associated to $A_\varphi$  given by
$$
\widetilde{A}_\varphi (u) = \frac{A_\varphi (u) }{||A_\varphi
(u)||}.
$$
Note that $\widetilde{A}_0$ has the  following hyperbolic fixed
points:  the ``most expanding"  eigenvector $v$ of $A$ is
attracting for $\widetilde{A}_0$ while the ``most contracting" eigenvalue $w$ of $A$  is
repelling for $\widetilde{A}_0$.
See Figure~\ref{fig:saddle-nodem}.
 These  hyperbolic points of $\widetilde{A}_0$ in
$\mathbb{S}^1$ have continuations $v_\varphi$ and $w_\varphi$ for
every small $\varphi>0$. Note that for $\varphi >0$ it holds $\widetilde A_\varphi (u)>u$, $u=v, w$. That is, the
graph of $\widetilde A_\varphi$ moves up as $\varphi$ increases.
 This implies that the fixed points $v_\varphi$
and $w_\varphi$ of $\widetilde A_\varphi$ ``move up and down,"
respectively, and hence $v< v_\varphi< w_\varphi< w$. Thus there is
some $\varphi_0$  (associated to a saddle-node bifurcation of the family
$(\widetilde A_\varphi )_{\varphi}$) such that these two points coincide.
The
parameter $\varphi_0$ corresponds to a real eigenvalue of
multiplicity two of $A_{\varphi_0}$. For $\varphi> \varphi_0$ close enough to
$\varphi_0$,
this leads to a pair of conjugate nonreal eigenvalues.
\end{proof}

\begin{figure}
\begin{overpic}[scale=0.1,
]{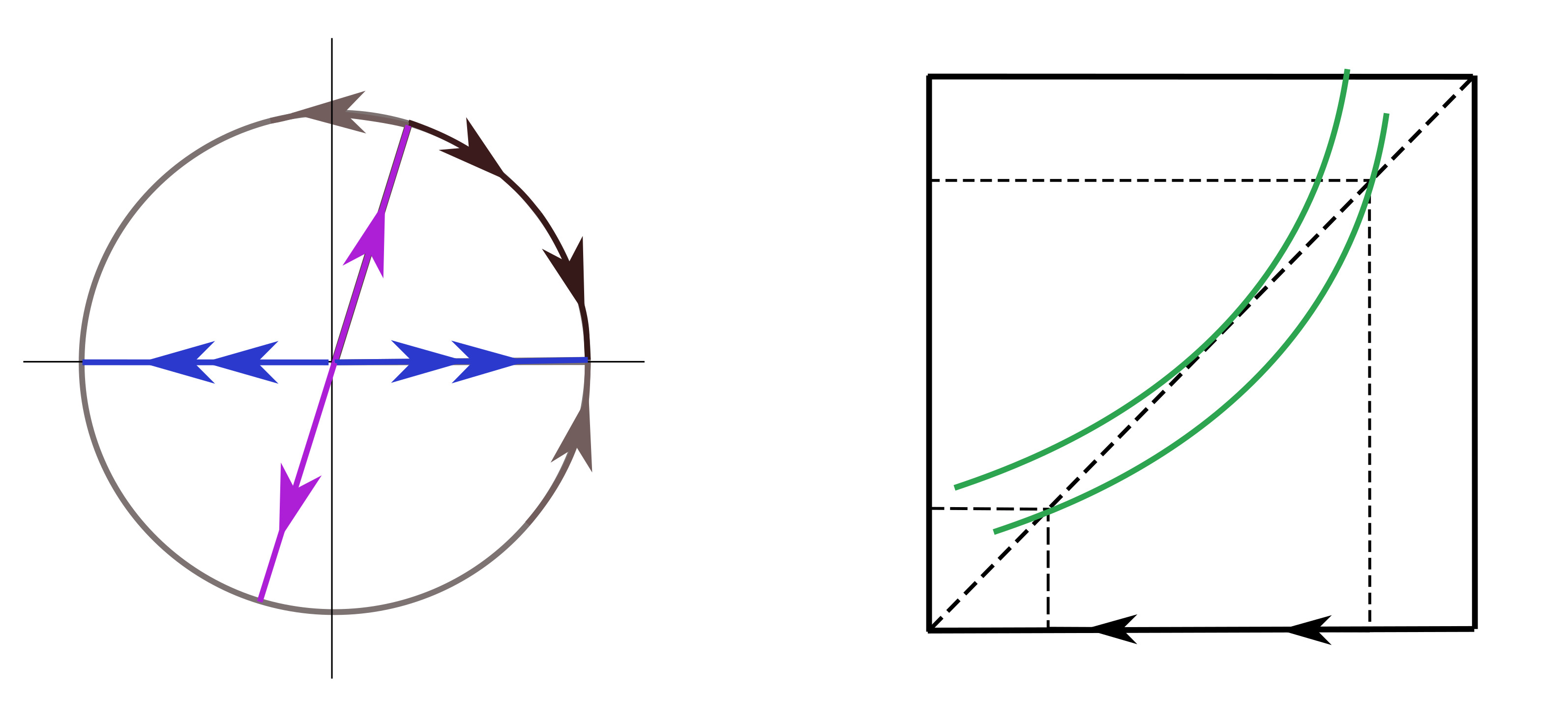}
  \put(130,82){\Large {$v$}}
       \put(117,72){\large {$\bullet$}}
      \put(-10,82){\Large {$-v$}}
            \put(14,72){\large {$\bullet$}}
   \put(84,135){\Large {$w$}}
          \put(81,120){\large {$\bullet$}}
    \put(32,8){\Large {$-w$}}
        \put(50,23){\large {$\bullet$}}
               \put(65,72){\large {$\bullet$}}
\put(223,86){$\widetilde A_{\varphi_0}$ }
\put(253,49){$\widetilde A_{0}$ }
      \put(210,8){\Large {$v$}}
            \put(275,8){\Large {$w$}}
       \put(211.5,42){{$\bullet$}}
        \put(277.5,108){{$\bullet$}}
 \end{overpic}
\caption{Saddle-node bifurcation}
\label{fig:saddle-nodem}
\end{figure}
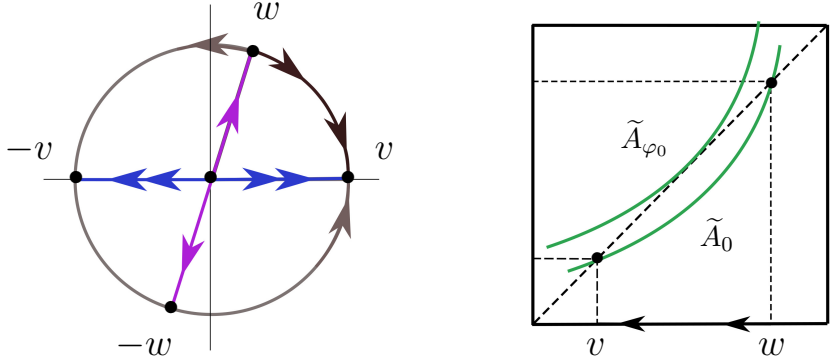

We now recall a standard construction involving transverse homoclinic points and hyperbolic sets.

\begin{remark}
\label{r.preliminaryconst} Consider $f\in
\mathrm{Diff}^r(M^{m+n})$ with  a pair of   saddles
$P$ and $P'$ of $u$-index
$n$ as in Proposition~\ref{p.Q}.
Consider open sets $U\ni P$  and $V\ni P'$, transverse homoclinic
points $X \in W^s(P, f) \pitchfork W^u (P',f)$ and $Y \in W^s(P',
f) \pitchfork W^u (P,f)$ with (recall the definitions in \eqref{e.localmanifolds})
\[
\begin{split}
&f^{k_X} (X) \in W^s_{U, \mathrm{loc}} (P,f), \quad f^{-k_X} (X)
\in W^u_{V,\mathrm{loc}} (P',f)
\\
&f^{k_Y} (Y) \in W^s_{V, \mathrm{loc}} (P',f), \quad f^{-k_Y} (Y)
\in W^u_{U, \mathrm{loc}} (P,f)
\end{split}
\]
for some constants $k_X, k_Y>0$, see Figure~\ref{fig:rel}. Then for every pair of  small
neighbourhoods $U_X$ of $X$ and $U_Y$ of $Y$, the maximal invariant
set   $\Lambda =\Lambda_{\widehat{V}}$  of $f$ in
$$
\widehat U  = U \cup V \cup \widetilde U_X \cup \widetilde U_Y,
\quad \mbox{where} \quad \widetilde U_Z = \bigcup_{i=-k_Z}^{k_Z}
f^i (U_Z), \quad Z=X,Y,
$$
is a transitive hyperbolic set
with
 a dominated splitting of the form
$$
 E^{ss} \oplus  E^{cs}  \oplus E^{cu} \oplus E^{uu}, \quad E^s=E^{ss}
\oplus E^{cs},
\quad E^u=E^{cu}
\oplus E^{uu},
$$
where the dimensions of $E^{cs}$ and $E^{cu}$
are $2$ and $n_u$, respectively.
This decomposition follows from the  type of
$P$  and the simple spectrum  of $P'$.

We observe that there is a homoclinic version (i.e., with $P$ having transverse
homoclinic points
and considering only one transverse homoclinic point) of the construction
above.
\end{remark}

\begin{proof}[Proof of Lemma~\ref{l.nonreal}]
Let $\Lambda=\Lambda_{\widehat{U}}$ be the hyperbolic
set provided by Remark~\ref{r.preliminaryconst}
(applied to $P$, $P'$,
$U$, $V$ as in the lemma, and the corresponding homoclinic points). We take a neighbourhood $U_0$
of $P$ contained in $U$, a transverse homoclinic point $A\in \Lambda$ of $P$, and $\ell>0$ such that the orbit $O_f (A)$ of $A$
satisfies
\begin{equation}
\label{e.transitionell}
\begin{split}
&\mathcal{O}_f(A) \setminus  \{f(A), \dots, f^{\ell-1} (A)\} \subset  U_0, \\
& \{f(A), \dots, f^{\ell-1} (A)\} \cap U_0 =\emptyset, \\ \
&\mathcal{O}_f(A) \cap V \ne\emptyset.
\end{split}
\end{equation}

\begin{figure}
\label{fig:rel}
\begin{overpic}[scale=0.1,
]{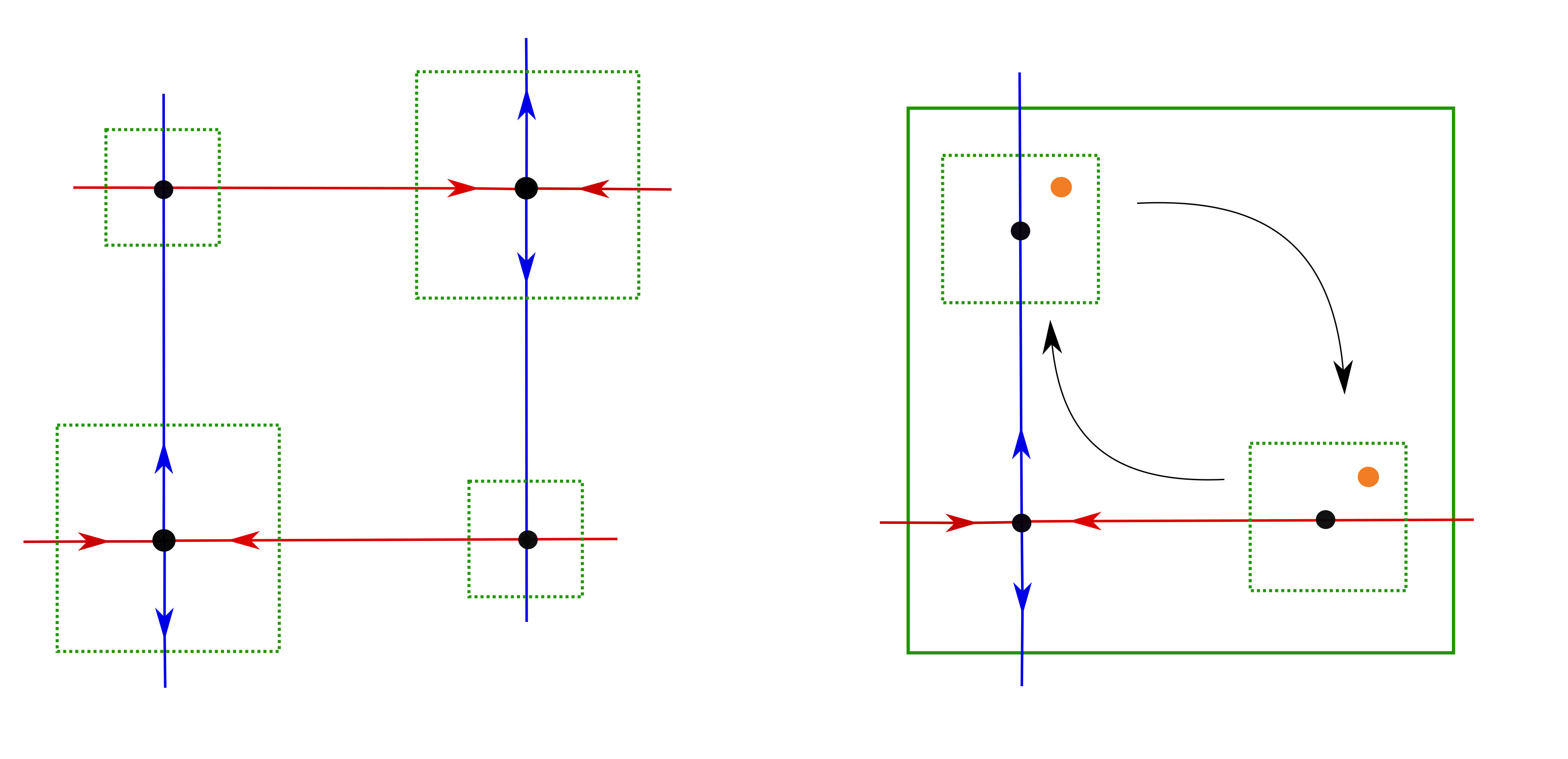}
  \put(17,85){\Large {$U$}}
       \put(47,60){\large {$P$}}
   \put(14,160){\Large {$U_Y$}}
       \put(45,144){\large {$Y$}}
 \put(100,75){\Large {$U_X$}}
       \put(131,60){\large {$X$}}
   \put(100,168){\Large {$V$}}
       \put(133,144){\large {$P'$}}
    \put(223,117){\large {$A$}}
            \put(250,140){\large {$f^k(Q_k)$}  }
          \put(215,147){\Large {$U_A$}}
       \put(297,39){\large {$f^{\ell}(A)$}  }
             \put(325,67){\large {$Q_k$}  }
     \put(310,127){\Large {$  \ffS $}  }
       \put(260,80){\Large {$ \ffL^k $}  }
           \put(217,32){\Large {$P$} }
 \end{overpic}
\caption{Generation of saddles with prescribed itineraries}
\end{figure}
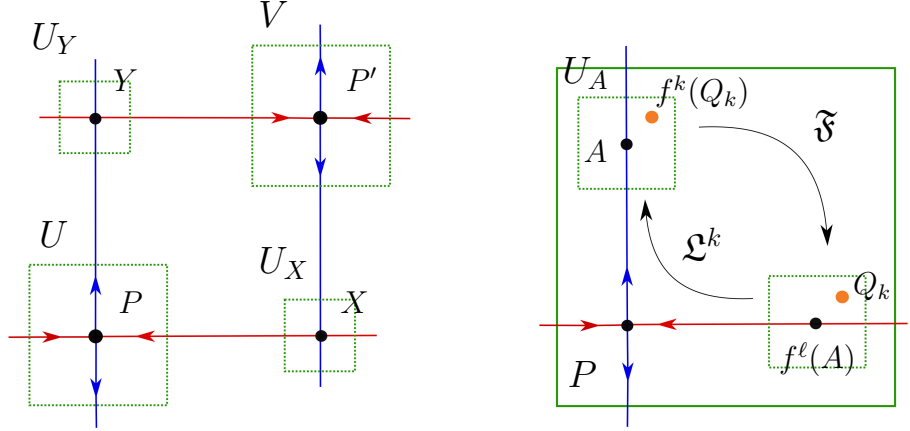

Take a small neighbourhood $U_A\subset U_0$ of $A$  such that
$f^{\ell}(U_A)\subset U_0$, $f^i (U_A) \subset \widehat U$ for
every $i\in \{0, \dots, \ell\}$,
 and the
sets $f(U_A), \dots , f^{\ell-1} (U_A)$ are pairwise  disjoint and
disjoint from $U_0$. Let $\ffF$ be the restriction of $f^{\ell}$
to $U_A$, that is, the  transition map from $U_A$ to $U_0$ along
the orbit of $A$, and $\ffL$ the restriction of $f$ to $U_0$ (a
linear map). See Figure~\ref{fig:rel}. Observe that there is $k_0$ such that for every
$k\geqslant k_0$ there is a periodic point $Q_k$ of period $k+\ell$ and
$u$-index $n$ such that
\begin{itemize}
\item $Q_k\in f^{\ell} (U_A)$,
\item
$f^i(Q_k)\in U_0$ for every $i=0, \dots, k$ and $f^k(Q_k)\in U_A$,
\item
the eigenvalues  of $Df^{k+\ell}(Q_k)$ satisfies (following the notation in \eqref{e.igualalambda})
$$
|\lambda^{(k)}_{3}| < |\lambda^{(k)}_{2}| \leqslant
|\lambda^{(k)}_1|
<1 < |\gamma^{(k)}_1| \leqslant  |\gamma^{(k)}_2|
< |\gamma^{(k)}_3|,
$$
\item
$Q_k$ is homoclinically related to $P$.
\end{itemize}
This implies that
$$
Q_k = f^{k+\ell} (Q_k) =   \ffS \circ \ffL^k (Q_k) \in \Lambda
$$
and that the orbit of $Q_k$ intersects both $U_0\subset U$ and $V$.

We first consider the case when $P$ is  of type $(2,1)$.
If
there are some large $k$ such that
 $\lambda^{(k)}_{2}$ and $\lambda^{(k)}_{1}$ are nonreal or equal, we are done.
 Thus, in what follows, we assume that   $ \lambda^{(k)}_{2}$ and $\lambda^{(k)}_1$
   are both real and consider  the two-dimensional central stable bundle
 $E^{cs}(Q_k)$ associated to them provided by Remark~\ref{r.preliminaryconst}.

We write the linearising coordinates of $P$  and the matrix
$Df(P)$ in the form
$$
(x_{m-2}, x_c, x_n) \in \mathbb{R}^{m-2} \times \mathbb{R}^2
\times \mathbb{R}^n, \qquad Df(Q)= \ffL=(L_{m-2} ,\lambda \,
R_\theta,  L_n),
$$
where $\lambda\in (0,1)$, $R_\theta$ is as in
\eqref{e.rotation}, and $L_{m-2}$ and $L_n$ are linear maps. We
consider linear maps of the form
\begin{equation}
\label{e.linearpart}
\ffL_{\alpha} =   (L_{m-2}, R_\alpha \circ (\lambda \, R_\theta) ,
L_n)= (L_{m-2}, \lambda R_{\alpha+\theta} , L_n)
\end{equation}
and the planes
$$
\pi \eqdef \{ (0^{m-2}, v_c, 0^n), v_c\in \mathbb{R}^2\}, \quad
\widetilde{\pi} \eqdef D\ffS_{\ffS(A)}^{-1} (\pi).
$$

We first claim that we can assume that the determinant of
$D\ffS{_A}|_{\widetilde{\pi}}$ is positive and
$D\ffS{_A}|_{\widetilde{\pi}}$  preserves the orientation.
Otherwise,
we consider new homoclinic points corresponding to two ``loops" of
$A$, that is, a new transverse homoclinic point  $A'$ of $P$ nearby
$A$ whose orbit is as follows: there is large even $i_0$ such that
\[
\begin{split}
&A'\in W^u_{U_0, \mathrm{loc}} (P), \quad f^{\ell+i_0+\ell}(A') \in W^s_{U_0, \mathrm{loc}} (P),\\
&f^{\ell}(A'), f^{\ell+i_0} (A') \in U_A, \\
& f^{\ell+i}(A') \in U_0 \quad \mbox{for every $i\in \{0, \dots
i_0\} $}.
\end{split}
\]
As for large $i_0$ the transition map $\ffS'$ associated to
$f^{2\ell+i_0}$ at $A'$ is close to $\ffS \circ \ffL^{i_0} \circ
\ffS$ our claim follows. Note that for this new point $A'$ one must
shrink the linearising neighbourhood in the next steps, but this
just adds additional iterations by a linear map that does not change
the central determinant.

To conclude the proof, consider local perturbations $f_\alpha$ of
$f$ whose derivative in $U_0$ is $\ffL_\alpha$. For each large $k$
and $\alpha \in [0,2\pi]$ the diffeomorphism $f_\alpha$ has a
periodic point $Q_{k,\alpha}$ that is the continuation of $Q_k$.
Note that if there is small $\alpha$ and large $k$ such that
$$
D (\ffL^k_\alpha \circ \ffS)_{\ffS^{-1} (Q_{k,\alpha })}= \ffL^k_\alpha \circ
\ffS_{{k,\alpha}}, \quad \ffS_{{k,\alpha}} \eqdef D
\ffS_{\ffS^{-1} (Q_{k,\alpha})}
$$
has a nonreal central eigenvalue, we are done.
Note that, independently of $\alpha$,
\begin{equation}
\label{e.nosemueve} \widehat Q_{k,\alpha} \eqdef \ffS^{-1}
(Q_{k,\alpha}) \to A \quad \mbox{and} \quad \ffS_{{k,\alpha}} \to
D \ffS_{A}, \quad \mbox{as $k\to \infty$.}
\end{equation}
We also have
$$
\ffL^k_{\alpha} =
 (L_{m-2}^k,  R_{k\alpha} \lambda^k R_{k\theta} , L_n^k)
= (\mathrm{Id}_{m-2},  R_{k\alpha}, \mathrm{Id}_n) \circ \ffL^k,
$$
where $\mathrm{Id}_r$ denotes the identity map in $\mathbb{R}^r$.
Fix large $k$ and the center eigenvectors of $v^k_\alpha$ and $w^k_\alpha$ of
$Df^{k+\ell}(\widehat{Q}_{k,0})= \ffL^k  \circ \ffS_{k,0}$.
For the parameter $\alpha=0$, we
can chose these vectors such that
$$
u^k_0 = (\tau^s_{k} (u^k_{c}),  u^k_{c}, \tau^u_k (u^k_{c}))\in
\mathbb{R}^{m-2} \times \mathbb{R}^2 \times \mathbb{R}^n, \quad
u=v, w,
$$
where $\tau^s_{k}\colon \mathbb{R}^2 \to \mathbb{R}^{m-2}$
and  $\tau^u_{k}\colon \mathbb{R}^2 \to \mathbb{R}^{n}$
go to the zero map as $k\to \infty$
and $v^k_c$ and $w_c^k$ are unitary  eigenvectors in $\mathbb{R}^2$
as in the proof of Claim~\ref{c.r.complex}.
For each $\alpha\in [0,2
\pi]$ we consider
$$
\widehat u^k_\alpha = D\ffS^{-1}_{Q_{k, \alpha}} (u^k)= (\widehat
\tau^s_{k} (\widehat u^k_{c,\alpha}),  \widehat u^k_{c,\alpha}, \widehat
\tau^u_k (\widehat u_{c, \alpha})), \quad
 \quad u=v,w,
$$
where the maps $\widehat{\tau}^{s,u}_k$ are defined similarly as above.
Observe that, by \eqref{e.nosemueve}, for each fixed $\alpha$ the variation of these
vectors goes to zero as $k\to \infty$.

We now observe that, for
small $\alpha$ and big $k$, the images of the vectors $v^k_0$ and
$w^k_0$ by $\ffL_\alpha^k$ are
the  rotation by the angle $k\alpha$ (that is big if $k$ is big) in the central direction
of their images by $\ffL^k$
(which are parallel to $\widehat v^k_0$ and $\widehat w^k_0$, respectively).
Thus these vectors move as in the proof of Claim~\ref{c.r.complex}. This
displacement is not altered by $D\ffS$ since this map preserves
the orientation in the central plane.
Now the conclusion follows
 as in Claim~\ref{c.r.complex}, implying that
$Q_{k,\alpha}$ is of type $(2,\ast)$, $\ast \in \{1,2\}$.
The fact that if $P$
is of type $(2,1)$ then $Q_{k,\alpha}$ is of type $(2,1)$
 follows from domination arguments.

 In the case when $P$ is of type $(2,2)$ the argument is analogous. The only difference is that we need to add a rotation-like term in the center unstable direction. Equation
 \eqref{e.linearpart} now reads as follows
 \[
 \begin{split}
\ffL_{\alpha, \beta} &=   (L_{m-2}, R_\alpha \circ (\lambda \, R_\theta) ,
R_\beta \circ (\gamma \, R_\varphi),
L_{n-2})\\
&= (L_{m-2}, \lambda R_{\alpha+\theta} ,
 \gamma \, R_{\beta+\varphi}
L_{n-2}),
\end{split}
\]
where $\gamma>1$.
This leads to a family depending on two independent parameters $\alpha$ and $\beta$. This allows us to deal with each two-dimensional central
direction independently, simultaneously getting  nonreal eigenvalues.

Finally, to get the property of the leading Jacobians, note that the number of iterates of the points $Q_{k,\alpha}$ in
the neighbourhood $U_0$ can be chosen arbitrarily large
compared to the length of the orbit. The assertion now
follows using the domination splitting. This ends the proof of the lemma.
\end{proof}

\begin{remark}[The case $n=1$] \label{r.longmanifolds}
In our construction, the orbit of the resulting
saddle $Q$ has most of  its iterates in the linearising neighbourhood
$U$ of $P$. Therefore, its local manifolds are large.
This implies that  $Q$ is homoclinically related to $P$ (and hence to $P'$).
Consider now the case  $n=1$. As the saddle $P$
has transverse homoclinic points, the property of $Q$ having large
unstable manifolds implies that both components of
$W^u_{U,\mathrm{loc}} (Q)\setminus \{Q\}$ transversely intersects
the
 stable manifold of
$P$ (and hence the stable manifold of $Q$, since  $P$ and
$Q$  are homoclinically related). This implies that $Q$ is
biaccumulated by its transverse homoclinic points.
\end{remark}

\begin{remark}\label{r.preservation}
Assume that the saddle point $P$ in Lemma~\ref{l.nonreal} has a
homoclinic tangency. Then the ``local" perturbation $g$ of $f$ can
be taken preserving that tangency.
\end{remark}

\subsection{Proof of Proposition~\ref{p.Q}}
We need to consider the following types $(2,1)$, $(2,2)$, and $(1,2)$ for $P$. We prove the proposition when $m_s=2$
(in the case $(1,2)$ just consider $f^{-1}$).
The
existence of the saddle $Q$ of type $(2,n_u)$ follows from
Lemma~\ref{l.nonreal}.
 The homoclinic relations and the biaccumulation property (when $n=1$)
 follow from Remark~\ref{r.longmanifolds}. Remark~\ref{r.preservation} implies the preservation of
 the tangency of $P$. This ends the proof of the proposition.

\section{Return maps at the homoclinic tangency}
\label{s.dynamicsatthetangency}

Throughout this section,  we consider $f\in \mathrm{Diff}^r(M^{m+n})$, $r,m \geqslant 2$ and
$n\geqslant 1$, having a homoclinic
tangency $Y$ associated to a  simple saddle $P$ of type
$(m_s,n_u)$  and $u$-index $n$.   We first state in Section~\ref{ss.genericc} generic conditions
{\bf{(C0)}}--{\bf{(C5)}}
at the homoclinic tangency. Thereafter,
we fix a neighbourhood of the tangency and introduce first return maps to it (Section~\ref{ss.returnmap22}). We also see how the expanding
leading
 Jacobian of the
saddle $P$ (see Equation~\eqref{e.J})
implies expansion properties of the return maps (Section~\ref{s.expandingprop}).
\subsection{Generic conditions}
\label{ss.genericc}
For simplicity, in what follows, we assume that the
period of $P$ is one.
We also assume that
\begin{enumerate}
\item[{\bf{(C0)}}] \label{lineal-coordenadas}  $f$ is $C^r$ linearisable at
$P$.
\end{enumerate}
By Sternberg linearisation theorem ~\cite{Ste:57},
every $C^r$ neighbourhood of $f$ contains diffeomorphisms satisfying this condition.

 We consider a small neighbourhood $W$ of $P$
 with these linearising coordinates
   \begin{equation}
 \label{e.linearisingcoor}
 (u,x,y,v)\in
\mathbb{R}^{m-m_s}\times\mathbb{R}^{m_s}\times
\mathbb{R}^{n_u}\times \mathbb{R}^{n-n_u}.
\end{equation}
In these coordinates, $P=(0^{m-m_s},0^{m_s},0^{n_u},0^{n-n_u})$
and the map $\ffT_0\eqdef f|_{W}$ is of the form
\begin{equation}\label{e.linearizacion}
(u,x,y,v)\,\, \mapsto\,\, \ffT_0(u,x,y,v)\eqdef(Au, Bx,Cy, Dv),
\end{equation}
here $A,B,D,C$ are square  matrices
whose eigenvalues  are $\lambda_m$,
$\lambda_{m-1},\dots,\lambda_{m_s+1}$ (the ones of $A$),
$\lambda_{m_s}, \dots , \lambda_{1}$ (the ones of $B$),
$\gamma_{1},\dots, \gamma_{n_u}$ (the ones of $C$), and
$\gamma_{n_u+1}, \dots, \gamma_{n}$ (the ones of $D$).
In the neighbourhood $W$ the local invariant sets are
(we omit the dependence on $W$)
\begin{equation}\label{e.localm}
\begin{split}
W^u_{\mathrm{loc}}(P,f)&=\{(u,x,y,v)\colon u=0^{m-m_s},\, x=0^{m_s}\}\subset W,\\
W^s_{\mathrm{loc}}(P,f)&=\{(u,x,y,v)\colon y=0^{n_u},\, v=0^{n-n_u}\}\subset W,\\
W^{ss}_{\mathrm{loc}}(P,f)&=\{(u,x,y,v)\colon x=0^{m_s},\,
y=0^{n_u},\, v=0^{n-n_u} \}\subset W.
\end{split}
\end{equation}

  We select
 points $Y^\pm$, with  $f^\ell(Y^-)=Y^+$ for some $\ell>0$,
 in the orbit of the tangency point $Y$ such that
 \begin{equation}\label{e.Transversality}
\begin{split}
&Y^-=(0^{m-m_s},0^{m_s},y^{-},v^{-} )\in W^u_{\mathrm{loc}}(P,f),\\
&Y^+=(u^+,x^+,0^{n_u},0^{n-n_u})\in W^s_{\mathrm{loc}}(P,f),\\
&\{ f(Y^-), \dots, f^{\ell-1} (Y^-)\}\cap W=\emptyset.
\end{split}
\end{equation}
This can be obtained by shrinking $W$.

To state the remaining conditions {\bf{(C1)}}--{\bf{(C5)}},
recall that if $m>m_s$ there exists the \textit{strong stable
foliation} $\mathcal{F}^{ss}$   of dimension $m-m_s$ defined on $ W^s(P,f)$.
This is the unique $f$ invariant foliation of dimension $m-m_s$
having the strong stable manifold $W^{ss}(P,f)$
of $P$ (tangent to
$E^{ss}(P)$) as a leaf. Similarly,  if $n>n_u$ there is the \textit{strong unstable
foliation} $\mathcal{F}^{uu}$
of dimension $n-n_u$
 defined on $ W^u(P,f)$.
In the coordinates above,
the (local) invariant foliations $\mathcal{F}^{ss}$ on $W^s_{\mathrm{loc}}(P,f)$
and $\mathcal{F}^{uu}$ on $W^u_{\mathrm{loc}}(P,f)$ are also straightened and have, respectively,  the form
\begin{equation*}
\begin{split}
&\{(u,x,y,v)\colon x=\mathrm{const.},\,y=0^{n_u},\, v=0^{n-n_u}\},\\
&\{(u,x,y,v)\colon u=0^{m-m_s},\, x=0^{m_s},\, y=\mathrm{const.} \}.
\end{split}
\end{equation*}

Let $E^{cu}(P)$ be the $Df$-invariant space of dimension $m_s+n$
corresponding to the eigenspaces associated to the eigenvalues of
modulus greater or equal than the stable leading
multiplier of $P$.
Following \cite{Tur:96} and based on \cite{HirPugShu:77},
an {\em{extended center unstable manifold $W^{cu}(P,f)$ of $P$}}
is an invariant manifold of dimension $m_s+n$ tangent to
$E^{cu}(P)$ at $P$  containing $W^u(P,f)$. This manifold always
exists, but in general, it is not uniquely defined.
Similarly, we can define
the {\em{extended center stable manifold $W^{cs}(P,f)$ of $P$}}
 of dimension $m+n_u$ tangent to
$E^{cu}(P)$ at $P$ and contains $W^u(P,f)$.

Conditions {\bf{(C1)}}--{\bf{(C5)}} are similar to the ones in
\cite{GonTurShi:93,GonTurShi:93b}, translating the quasi-transverse conditions in
\cite{NewPalTak:83} to our higher-dimensional context.
For that, consider the subspace
\begin{equation}\label{e.EY}
E_Y \eqdef T_Y W^s(P,f)\cap T_Y W^u(P,f).
\end{equation}
The required conditions are:
\begin{itemize}
\item[{\bf{(C1)}}]
$ \dim  (E_Y) = (m+n) - \dim (T_Y W^s(P,f)+ T_Y W^u(P,f))=1,$
\item[{\bf{(C2)}}]
 $E_Y$ is not contained in $T_Y \mathcal{F}^{ss}
(Y)$, where $\mathcal{F}^{ss} (W)$ denote the leaf of
$\mathcal{F}^{ss}$ through $W$,
\item[{\bf{(C3)}}]
the contact between the stable and unstable manifolds at the
tangency point $Y$ is quadratic,
\item[{\bf{(C4)}}]
$W^{cu}(P)$  is transverse to $\mathcal{F}^{ss} (Y^+)$ at $Y^+$, and
\item[{\bf{(C5)}}]
{$W^{cs}(P)$  is transverse to $\mathcal{F}^{uu} (Y^-)$ at $Y^-$.}
\end{itemize}
Conditions {\bf{(C4)}} and {\bf{(C5)}} are intrinsic as they do not depend on the
choice of the extended center manifolds.
Moreover, when $n=1$ (and hence $n_u=1$) condition {\bf{(C5)}} is vacuous.
Finally, our conditions are $C^r$ generic.

\begin{remark}
\label{r.setLambda} Consider the $f$-invariant compact set
$\Lambda$ formed by the orbits of the saddle $P$ and the tangency
$Y$. Under conditions  {\bf{(C1)}} -- {\bf{(C4)}},   the set
$\Lambda$ has a dominated splitting
 of the form
\begin{equation}\label{e.split}
T_\Lambda M= E^{ss}\oplus E^{cu}\quad \mbox{with}\quad \dim
(E^{cu})=m_s+n,
\end{equation}
where $E^{ss}$ is uniformly contracting. We consider cone
fields
\begin{equation}
\label{e.cc} X\in U\to\mathcal{C}^{ss}(X)\qquad \mbox{and}\qquad
X\in U\to \mathcal{C}^{cu}(X)
\end{equation}
defined on a small neighbourhood $U$ of $\Lambda$ such that
$E^{ss}\subset \mathcal{C}^{ss}$ and $E^{cu}\subset
\mathcal{C}^{cu}$ as
 in Proposition~\ref{prop:cone}.
 The set  $U$ is called a {\em{partially hyperbolic neighbourhood}} of $\Lambda$.

 A similar comment holds when  {\bf{(C4)}} is replaced by {\bf{(C5)}}.
\end{remark}

\begin{remark}
\label{r.cgd} The strong stable foliation $\mathcal{F}^{ss}$ of
$W^s(P,f)$ can be smoothly extended to a $C^r$ foliation
$\widetilde{\mathcal{F}}^{ss}$ defined on a neighbourhood of
$\Lambda$ (see \cite{Tur:96}). In particular,  after shrinking the
partially hyperbolic neighbourhood $U$, we can assume that the
leaves
 $\widetilde{\mathcal{F}}^{ss}(X)$, $X\in U$,
 are
 tangent to the cone-field $\mathcal{C}^{ss}(X)$ and
contain disks of uniform size. Finally, the restriction of this
foliation to the set $\Lambda$ is invariant.
\end{remark}

 Throughout this section, we assume that the saddle $P$
 is of type $(m_s,n_u)$
  and that $P$ and  its homoclinic tangency $Y$ satisfy conditions  {\bf{(C0)}}--{(\bf{C5)}}.

 \subsection{Return maps} \label{ss.returnmap22}
We take small closed neighbourhoods $\Pi^{\pm}\subset W$ of
$Y^\pm$
(called {\em{reference boxes}})
such that $f (\Pi^+)\cap
\Pi^+=\emptyset=
 f^{-1}(\Pi^-)\cap \Pi^-$.
Following
 \cite[Corollary 1]{GonShiTur:08} (see also~\cite{GonTurShi:93,GonTurShi:93b}), under
 conditions~{\bf{(C1)}}--{\bf{(C5)}},
 the Taylor expansion
  of the {\em{transition map}} $\ffT_1$ from $Y^-$ to $Y^+$ can be written as follows:
\begin{equation*}
\ffT_1\eqdef f^{\ell}|_{\Pi^-} \colon \Pi^- \to
\Pi^+,
\qquad
\ffT_1(u,x,y,v)= (\bar{u},\bar{x},\bar{y},\bar{v})
\end{equation*}
where
  \begin{equation}
 \label{e.defT1}
\begin{split}
 \bar u - u^+&=A_{1}\,  u +B_{1 }\, x+C_{1 }\, (y-y^-)+
 D_{1 }\, \bar{v} + \dots
 ,\\
 \bar x -x^+ &= A_{2 }\,u +B_{2 }\,x+C_{2 }\, (y-y^{-})
+  D_{2 }\, \bar{v}+ \dots
   , \\
 \bar y &=A_{3 }\, u + B_{3 }\, x+C_{3 }\,(y-y^{-})^2+
  D_{3 }\,\bar{v}+ \dots  ,
   \\
 v -v^-&=  A_{4 }\, u + B_{4 }\, x+C_{4 }\,(y-y^{-})+
  D_{4 }\, \bar{v} + \dots    ,
\end{split}
\end{equation}
where
 the notation $(\dots)$ stands for higher order  terms and
$A_{i}, B_i, C_{i},D_i$ are matrices of the appropriate dimensions. Here $B_3$, $C_2$
 and $C_3$ are nonzero matrices.

To study the returns associated to $\ffT_1$,  consider for sufficiently large $k$, say $k> k_0$, the
{\em{strips}}
 \begin{equation}
 \label{e.defPik}
 \begin{split}
  \Pi_k^+&\eqdef
   \ffT_0^{-k}(\Pi^-) \cap \Pi^+ \\
  &=\{x\in \Pi^+ \colon \ffT_0^{i} (x) \in W \mbox{ for all $i=0, \dots, k$ and $\ffT_0^k(x) \in \Pi^-$} \},
  \\
  \Pi_k^- &\eqdef  \ffT_0^k (\Pi_k^+) \subset \Pi^-.
 \end{split}
 \end{equation}
The \textit{return map} at the tangency point $Y^+$ is defined in
a subset ${\Sigma}^+ \subset \Pi^+$ by
\begin{equation}\label{e.return0}
\ffR \colon {\Sigma}^+  \eqdef \bigcup_{k>k_0}\Pi^+_k \to
\Pi^+,\quad  \ffR|_{\Pi^+_k}=\ffR_k \eqdef  \ffT_1 \circ \ffT_0^k.
\end{equation}

 We start analysing suitable ``first return maps''
defined on a neighbourhood of the homoclinic tangency $Y^+$.

\subsection{Expanding returns  along center unstable cone fields}
\label{s.expandingprop}

We now obtain``ex\-pan\-sion of volume"  properties (along the leading directions)
derived from the
expanding leading Jacobian of the saddle $P$, see Section~\ref{sss.expansionvalume}.
Additionally, in the codimension one case, we get
expansion of suitable diameters, see Section~\ref{sss.expdiameters}.

\subsubsection{Expansion of volume}
\label{sss.expansionvalume}
Next lemma about expansion of  volume of local submanifolds tangent
to the cone field
$\mathcal{C}^{cu}$
in Remark~\ref{r.setLambda}
by the return maps $\ffR_k$
is a variation of~\cite[Lemma~5]{LiTur:20}. To
state it recall the definition of $J_P=J_P(f)$
in~\eqref{e.J} and that a
submanifold $\cS$ is tangent to $\mathcal{C}^{cu}$
if $T_X \cS \subset  \mathcal{C}^{cu}(X)$ for every $X\in \cS$.

\begin{lem} \label{lema1}
There is a constant $L>0$ such that for every sufficiently large
$k$ and every submanifold $\cS \subset {\Pi}_k^+$ of dimension
$m_s+n$ tangent to $\mathcal{C}^{cu}$ it holds
$$
   \vol(\,{\ffR}_k(\cS)\,) >  L \, J_P^k \,
   \vol(\cS),
$$
where $\vol(\cdot)$ stands for the $(m_s+n)$-dimensional volume.
\end{lem}

\begin{proof}
We use the linearising coordinates $(u,x,y,v)$ of points on the set $W$ in
\eqref{e.linearisingcoor}. For points $X$ close to $Y$
  the cone $\mathcal{C}^{cu}(X)$ is an open Grassmannian neighbourhood of the subspace
$\{0^{m-m_s}\}\times
\mathbb{R}^{m_s}\times \mathbb{R}^{n}$.
We can assume
that $\mathcal{C}^{cu}(X)$ does not depend on the point $X$, and hence
it is constant in a neighbourhood of $Y$. Thus, every $(m_s+n)$-dimensional $C^1$ manifold
 $\mathcal{S}\subset {\Pi}_k^+$
tangent to $\mathcal{C}^{cu}$  can be seen as the graph of a $C^1$
map in the variables $(x,y,v)$. Moreover,
 we have that\footnote{The
notation $f \asymp g$  means that there are constants  $C,D>0$
such that $C\,|g|< |f|< D\, |g|$.}
$$
\vol(\mathcal{S}) \asymp \vol(\,\pi(\mathcal{S})\,)
$$
where $\pi$ is the canonical projection from $W$ on $\{u=0^{m-m_s}\}$. For simplicity, we identify
$\{u=0^{m-m_s}\}$ with $\mathbb{R}^{m_s}\times \mathbb{R}^{n_u}\times
\mathbb{R}^{n-n_u}$ and write $\pi(u,x,y,v)=(x,y,v)$.

Since, by construction, $\mathcal{C}^{cu}$ is $D\ffR_k$-invariant,
we have that $\ffR_k(\mathcal{S})$ is also tangent to
$\mathcal{C}^{cu}$. Thus, we also have that
$$
\vol(\, \ffR_k(\mathcal{S})\,) \asymp
\vol(\,\pi(\ffR_k(\mathcal{S}))\,).
$$
Thus, to prove the lemma it is enough  to get a
constant $\widetilde K>0$ (independent of large  $k$) such that
\begin{equation}\label{e.fomi}
\vol(\,\pi(\ffR_k(\mathcal{S}))\,) \geqslant \widetilde  K \,
J_P^k \, \vol(\,\pi(\mathcal{S})\,).
\end{equation}

As $\pi |_{\mathcal{S}}\colon \mathcal{S} \to \pi(\mathcal{S})$
is a $C^1$
diffeomorphism, the map
$$
\widehat{\ffR}_k\colon \pi(\mathcal{S}) \to
\pi(\ffR_k(\mathcal{S})), \quad \widehat{\ffR}_k\eqdef \pi \circ
\ffR_k \circ (\pi|_{\pi(\mathcal{S})})^{-1}
$$
is a $C^1$ diffeomorphism
satisfying
\begin{equation} \label{comuta}
    \pi \circ \ffR_k = \widehat{\ffR}_k \circ \pi \quad \mbox{on}\quad  \mathcal{S}.
\end{equation}
Thus, we have that
$$
  \vol(\,\pi(\ffR_k(\mathcal{S})\,)=\vol(\,\widehat{\ffR}_k(\pi(\mathcal{S}))\,)=\int_{\pi(\mathcal{S})}
  \left|\det\frac{\partial
\widehat{\ffR}_k}{\partial(x,y, v) }\right| \, dx\, dy\, dv.
$$
 To obtain the  inequality \eqref{e.fomi}
 it is enough to get a constant $\widetilde K>0$ (independent of $k$)
such that
\begin{equation}
\label{e.conesobasta} \left|\det \frac{\partial
\widehat{\ffR}_k}{\partial(x,y,v)
}\bigg|_{\pi(\mathcal{S})}\right|\geqslant \widetilde K \, J_P^k .
\end{equation}

To get the inequality in \eqref{e.conesobasta},  recall that
$\ffR_k=\ffT_1\circ \ffT_0^k$, where
\begin{equation*}
\begin{split}
(u,x,y,v)\in \Pi_k^+\,\, &\mapsto\,\,
\ffT_0^k(u,x,y,v)=(u_k,x_k,y_k,v_k)\in \Pi^-,
\\
(u_k,x_k,y_k,v_k)\in \Pi^- \,\, &\mapsto\,\,
\ffT_1(u_k,x_k,y_k,v_k)= (\bar{u}_k,\bar{x}_k,\bar{y}_k,\bar{v}_k)\in
\Pi^+.
\end{split}
\end{equation*}
are given by the Equations~\eqref{e.linearizacion}
and~\eqref{e.defT1}. In view of Equation~\eqref{comuta}, we have
that $\widehat \ffR_k = (\pi \ffT_1 \pi^{-1})
\circ (\pi \ffT_0^k \pi^{-1})$
(here the maps $\pi^{-1}$ are local inverses in the corresponding submanifolds)
 and thus
\begin{equation}
\label{e.determinante}
  \frac{\partial \widehat \ffR_k}{\partial (x,y,v)}\bigg|_{\pi(\mathcal{S})}=
  \frac{\partial(\bar{x}_k,\bar{y}_k,\bar{v}_k)}%
  {\partial (x_k,y_k,v_k)}\bigg|_{\pi\ffT_0^k(\mathcal{S})} \, \,%
  \frac{\partial(x_k,y_k,v_k)}{\partial(x,y,v)
  }\bigg|_{\pi(\mathcal{S})}.
\end{equation}
Since, see~\eqref{e.linearizacion},
$$
x_k = B^k x, \quad
y_k = C^{k}y, \quad \text{and} \quad v_k=D^k v,
$$
where $B$, $C$, and $D$ are square matrices,
whose eigenvalues in absolute value are all equal to $\lambda$,
$\gamma$, and strictly greater than $\gamma$, respectively.
Thus, as
$$
|\det D | > \gamma^{n-n_u}>1
$$
it follows
\begin{equation}\label{e.exponential}
 \left|\det
\frac{\partial(x_k,y_k,v_k)}{\partial(x,y,
v)}\bigg|_{\{u=0^{m-m_s}\}}\right| =  J_P^k\, |\det D^k|\geqslant
 J_P^k.
\end{equation}

\begin{claim}
\label{cl.t1formula} There is $\widetilde K>0$ such that for every
$k$ large enough it holds
$$
\left|\det \frac{\partial(\bar{x}_k,\bar{y}_k,\bar{v}_k)}%
  {\partial (x_k,y_k,v_k)}\bigg|_{\pi  (\ffT_0^k(\mathcal{S}))}\right| > \widetilde K.
$$
\end{claim}
\begin{proof}
First, observe that $\det D\ffT_1 |_{\{u=0^{m-m_s\}}} (Y^-) \ne 0$. Otherwise,
the tangent space to $W^{cu}(P)$ at $Y^+$ (which is the image of
$\{ u=0^{m-m_s} \}$ by $D\ffT_1$) cannot be transverse to
$\mathcal{F}^{ss}(Y^+)$ at $Y^+$, contradicting assumption
{\bf{(C4)}}. In particular, $\det D\ffT_1 |_{\{u=0^{m-m_s}\}} \not= 0$ in a
neighbourhood $V$ of the tangency point $Y^-$.
Note that $\mathcal{S}\subset \Pi^+_k$ and $\ffT_0^k(\Pi^+_k) \subset
\Pi^-$ converges to $\{u=0^{m-m_s}\}$ as $k\to \infty$. Taking, if necessary, $\Pi^-$ small
 and $k$ large, we can guarantee that
$\ffT_0^k(\Pi^+_k)\subset V$. Then
$$
\det \frac{\partial(\bar{x}_k,\bar{y}_k,\bar{v}_k)}
  {\partial (x_k,y_k,v_k)}
  \bigg|_{\pi ( \ffT_0^k(\mathcal{S})) } \not =0.
 $$
 This concludes the proof.
  \end{proof}

Equation~\eqref{e.conesobasta} (and the lemma)  follows
from \eqref{e.determinante}, \eqref{e.exponential}, and
Claim~\ref{cl.t1formula}.
\end{proof}

\subsubsection{Expansion of diameters}
\label{sss.expdiameters} We now restrict our attention to the case
when  the saddle $P$ of $f$ has
$u$-index one (in particular,  $n=n_u=1$). Note that in this case, the
variables $v, \bar v$ in Equation~\eqref{e.linearisingcoor} do
not appear.

Using the linearising coordinates in~\eqref{e.linearisingcoor}, consider the  natural projections
\begin{equation}
 \label{e.proyecciones}
 \pi \colon W  \to \mathbb{R}^{m_s+1},\quad
 \pi_{c} \colon \mathbb{R}^{m_s+1}\to  \mathbb{R}^{m_s},
 \quad
 \pi_u \colon \mathbb{R}^{m_s+1}\to  \mathbb{R}.
 \end{equation}
Given a subset $A\subset (\{0^{m-m_s}\} \times
 \mathbb{R}^{m_s+1}) \cap W$
we define
\begin{equation}\label{e.diameter}
\diam_i(A)\eqdef \mathrm{diameter}(\pi_i(\pi(A)))\qquad  i=c,
\,u.
\end{equation}

\begin{lem}
\label{l.fromRussiawithconstants} There is a constant $K>0$,
depending only on the transition $\ffT_{1}$,
 such that
for every sufficiently large $k$ and every local submanifold $\cS\subset {\Pi}_k^+$
 of dimension $m_s+n$ tangent to
$\mathcal{C}^{cu}$ it holds
$$
\diam_c (\ffR_k (\cS) )<  K \, \gamma^k \, \diam_{u} (\cS),
$$
where $\gamma>1$ is the unstable leading
eigenvalue of $P$.
\end{lem}

\begin{proof}
Consider any $\cS$ as in the lemma and a point $X=(u,x,y)\in \cS \subset {\Pi}_k^+$. Let
$$
X_k=(u_k,x_k,y_k)\eqdef\ffT_0^k(X)\in \Pi^- \quad
 \widetilde X_k= \ffT_{1} (\ffT_0^k(X))=\ffR_k (X) \in \Pi^+.
$$
From~\eqref{e.linearizacion}, the coordinates $u_k$, $x_k$ are at
most of order of $\lambda^k$, where $\lambda<1$ is the leading stable
eigenvalue of $P$ in \eqref{e.igualalambda}.
Thus, by the formula of $\ffT_{1}$ in~\eqref{e.defT1} (where now the coordinate $v$ does not appear) the
$x$-coordinate of $\widetilde X_k$ is of the form
\begin{equation}\label{e.contribution}
x^+ +A_{2} u_k + B_{2}  x_k+C_{2} (y_k-y^{-})+\dots,
\end{equation}
 where $x^+$ is the $x$-coordinate of $Y^+$, $y^-$ is the $y$-coordinate of
$Y^-$, see Equation~\ref{e.Transversality}.
Note that Equation~\eqref{e.contribution} implies that there is a
constant $K>0$, independent of $\cS$ and large $k$, such that
\begin{equation}\label{e.penultima}
\diam_c (\ffR_k (\mathcal{S})) < K \, \diam_{u}
(\ffT_0^k(\mathcal{S})) = K\, \gamma^k\,\diam_{u} (\cS),
\end{equation}
where the last equality follows from  the linearising hypothesis.
 This completes the proof of  the lemma.
\end{proof}

\begin{remark}[Uniform estimates]
\label{r.uniform}
The constants $K$ and $L$ in Lemmas~\ref{lema1} and \ref{l.fromRussiawithconstants}
can be chosen uniform for perturbations. To see why this is so just note that the estimates in the proofs of these lemmas
only depend on the center unstable cone field (which persists by perturbations).
\end{remark}

\subsection{Return maps at the homoclinic tangency: the codimension one case}
\label{s.returnmapcoidmensionone}
We now refine the constructions in Section~\ref{ss.returnmap22}
in the codimension one case, when  the saddle $P$ has
$u$-index one and hence $n=n_u=1$.
Recall  the choices
of the tangency points $Y^-$ and $\ffT_1(Y^-)=Y^+$
in \eqref{e.Transversality},
the reference boxes $\Pi^\pm$, and
the strips $\Pi_k^\pm$ in \eqref{e.defPik}.

Assuming that the saddle $P$ is
biaccumulated by its transverse homoclinic points (recall the definition in the preamble of
Section~\ref{s.dominatedhorseshoed})
we will ``resize"  the sets $\Pi^+_k$  to
get additional dynamical control. We now go to the details.
Our construction is inspired by
 \cite[Section 1.4]{GonShiTur:08}
(see also previous constructions in \cite{New:79, PalVia:94}).

\begin{lem}
\label{lemma:herradura} Consider  $f\in \mathrm{Diff}^r(M^{m+1})$ with a saddle $P$ of  $u$-index one that is
biaccumulated by its transverse homoclinic points and has a homoclinic tangency.
Assume that conditions {\bf{(C0)}} -- {\bf{(C4)}} hold.
Then the reference box $\Pi^-$ can be chosen such that
there are sequences of
points $(N^-_j)$ and $(N^+_j)$ and of  $m$-dimensional submanifolds
$(\Theta_j)$ with
$$
N^\pm_j\in W^s(P,f)\pitchfork
W^u_{\loc}(P,f), \quad
N^\pm_j \in \Theta_j \subset  W^s(P,f)\cap \Pi^-,
$$
  such that:
\begin{itemize}
\item
the open ``interval'' $(N^-_j, N^+_j)\subset
W^u_{\mathrm{loc}}(P, f)$ contains $Y^-$ and satisfies
$$
\delta_j\eqdef\diam([N^-_j,N^+_j])\to 0 \quad \mbox{as} \quad j\to \infty,
$$
\item
each submanifold
$\Theta_j$
splits the set $\Pi^-$ into two connected components such that
the component $C_j$ of $(\Pi^- \setminus \Theta_j)$ containing  $Y^-$ satisfies
$C_{j+1} \subset C_j$.
\end{itemize}
\end{lem}
\begin{proof}
The codimension one hypothesis implies that $W^s_{\loc}(P,f)$ splits the
box $\Pi^+$
 into two connected components.
 So the same holds for $\Pi^-$ and $\ffT_1^{-1} (W^s_{\loc}(P,f))$. By the biaccumulation property, both components contain disks of
 $W^s(P,f)$ accumulating to $\ffT^{-1} (W^s_{\loc}(P,f))$. The disks in one of these components
 also have two transverse intersections with $W^u_{\mathrm{loc}}(P,f)$ accumulating to $Y^-$.
 This provides the sequences of submanifolds $\Theta_j$ and of transverse homoclinic points $N_j^\pm$.
\end{proof}

Fix large $j$ and recall
the definition of
$\Pi_k^\pm$ in \eqref{e.defPik}.
 For every $k$
large enough, the set $\Pi^-_k \setminus \Theta_j$
has three connected components. Denote by ${\Pi}^-_{j,k}$ the closure
of the connected component intersecting $C_j$  and let

 \begin{equation}
 \label{e.defPik2}
 \Pi_{j,k}^+ \eqdef
   \ffT_0^{-k}(\Pi^-_{j,k}) \cap \Pi^+  \subset \Pi^+_k.
 \end{equation}
Define the {\em{$s$-} } and  {\em{$u$-boundaries}}
 of  ${\Pi}_{j,k}^+$
by
\begin{equation}
\label{e.bordes}
\partial^s
\Pi^+_{j,k}\eqdef     \ffT_0^{-k}(\Theta_j\cap \Pi_k^-)\cap
\Pi^+_k
\qquad
\partial^u \Pi^+_{j,k}
\eqdef \overline{\partial  \Pi^+_{j,k}\setminus
\partial^s \Pi^+_{j,k}}.
\end{equation}
Observe that, by construction, $\partial^s  \Pi^+_{j,k}\subset W^s(P,f)$.

Recall the definitions of $\gamma>1$ in \eqref{e.igualalambda} and
of the $c$- and $u$-diameters in~\eqref{e.diameter}. Next claim
asserts that $c$-diameters of $\Pi_k^+$ are ``independent" of $k$
while $u$-diameters
 exponentially shrink.
The proof  is straightforward and hence omitted.

\begin{claim}[$c$- and $u$-diameters of $\Pi_{j,k}^+$]
\label{cl.inextremis} There is  $\rho>0$ such that  for every sufficiently large $j$
there is $k_0(j)$
 such that for every
 $k \geqslant k_0(j)$
it holds
$$
\diam_c ( \Pi^+_{j,k}) >\rho \quad \mbox{and} \quad
\diam_u({\Pi}^+_{j,k}) < \gamma^{-k} \, \delta_j.
$$
\end{claim}

\begin{remark}[Choice of quantifiers]\label{r.choicesdelta}
Recall the constants $K$, $\delta_{j}$, and $\rho$ in
Lemmas~\ref{l.fromRussiawithconstants}  and ~\ref{lemma:herradura} and Claim~\ref{cl.inextremis}, respectively. There is  large  $j_*\geqslant 1$ such that
$$
K \delta_{j_*} < \frac{\rho}{10 }.
$$
As in what follows $j_*$ remains fixed, for notational simplicity, we
write
\begin{equation}
\label{e.choicesdelta}
 \widetilde \Pi^\pm_k \eqdef \Pi^\pm_{j_*,k},
 \quad \mbox{where} \quad
    k>k_0(j_*).
\end{equation}
\end{remark}

\section{Unfolding the homoclinic tangency}
\label{s.indexvariation}

Throughout this section, we consider $f\in \mathrm{Diff}^r(M^{m+n})$,  $r,\, m \geqslant 2$
and $n
\geqslant 1$, with a
simple saddle $P$ of type
$(m_s,n_u)$ and $u$-index $n$
having a homoclinic
tangency $Y$
satisfying {\bf{(C0)}}--{(\bf{C5)}}.
We will study  how the unfolding of this tangency
generates saddles with
 $u$-indices different from the one of $P$.

\subsection{Unfolding families and return maps}
\label{ss.unfoldingreturn}
We embed the diffeomorphism $f$
 into a family $(f_{\mathbf{t}})$, with $f_{\mathbf{0}}=f$,
 unfolding the homoclinic tangency.
 For that, we borrow the construction from
\cite[Section 1.2]{GonShiTur:08}.
The number  of parameters of the
family $(f_{\mathbf{t}})$ depends on the type of the saddle $P$
  (which is fixed for every $f_{\mathbf{t}}$).
  We use the notation in Section~\ref{ss.returnmap22}.
 Note that  transition map $\ffT_{1, \mathbf{t}}\eqdef f_{\mathbf{t}}^\ell$, see \eqref{e.defT1},
   is defined for every small $\mathbf{t}$.
   Recall that $\lambda$ and $\gamma$ defined as in \eqref{e.igualalambda}  are the
   modulus of the contracting and expanding leading multipliers of $P$.

Using the
{\em{effective dimension,}} denoted by $\mathrm{d}_{\mathrm{e}}$,  introduced
in \cite{Tur:96}, there are three classes of simples saddles with leading Jacobian greater than one:
\begin{itemize}
\item
$P$ has  $\mathrm{d}_{\mathrm{e}}=1$ if it is of type $(1,1)$ or $(1,2)$ and $\lambda\gamma>1$;
\item
 $P$ has  $\mathrm{d}_{\mathrm{e}}=2$ if it is of type $(1,2)$ with $\lambda\gamma<1$ or of type $(2,1)$ or $(2,2)$ with $\lambda^2\gamma>1$;
\item
 $P$ has  $\mathrm{d}_{\mathrm{e}}=3$  if it is of type  $(2,2)$ and $\lambda^2\gamma<1$.
\end{itemize}}

The effective dimension is the number of parameters considered in the unfolding of the
tangency. Let us roughly explain the choice of  these parameters.
When  $\mathrm{d}_{\mathrm{e}}=3$ the parameter
$\mathbf{t}=(t,\alpha, \beta)\in [-\varepsilon, \varepsilon]^3$
(in cases  $\mathrm{d}_{\mathrm{e}}=2$  and $\mathrm{d}_{\mathrm{e}}=1$ we consider
 $\mathbf{t}=(t,\alpha)\in [-\varepsilon, \varepsilon]^2$
and $\mathbf{t}=t\in [-\varepsilon, \varepsilon]$, respectively) is defined as follows.
 Let\footnote{In this description, we omit
 the neighbourhood in the definition of  the local manifold.}
 $\omega^u_{\mathbf{t}} \eqdef W^u_{\mathrm{loc}} (P,f_{\mathbf{t}}) \cap \Pi^-$,
 then
 the distance between the
 folding point of
 $\ffT_{1, \mathbf{t}} (\omega^u_{\mathbf{t}})$
  and $W^s_{\mathrm{loc}}(P,f_{\mathbf{t}})$ is  $|t|$.
 To describe   $\alpha$ note that
 in this case the leading  multipliers
 of $P$ for $f$
are nonreal and equal to
 $\lambda \, e^{\pm i\theta}$ and
 $\gamma \, e^{\pm i \rho}$. The leading multipliers
of $P$ for $f_{\mathbf{t}}$ are
 $\lambda \, e^{\pm i(\theta+\alpha)}$ and
 $\gamma \, e^{\pm i (\rho+\beta)}$.
 We let $\ffT_{0,\mathbf{t}}= f_{\mathbf{t}}|_W$
 (here $W$ is the linearising neighbourhood in Section~\ref{ss.genericc})
 and note that
 $\ffT_{0,\mathbf{t}}$ is a
``double rotation'' of  $\ffT_{0}$ by angles $\alpha$ and $\beta$.

Note that in the previous construction
$P_{f_{\mathbf{t}}}=P$ and that
the leading Jacobian of $P$ is not modified:
\begin{equation}
\label{e.constantleadingJ}
J_P (f_{\mathbf{t}})=
J_P(f).
\end{equation}

Recall the
construction of the return maps $\ffR_k=\ffT_1 \circ \ffT_0^k$
in Section~\ref{ss.returnmap22}. We now consider
return maps
$\ffR_{k, \mathbf{t}} =\ffT_{1, \mathbf{t}} \circ \ffT_{0, \mathbf{t}}^k$
similarly defined.
 For large $k>0$,
 as in \eqref{e.defPik}, we consider the strips
 $$
 \Pi_{k,  \mathbf{t}}^+ \eqdef \ffT_{0, \mathbf{t}}^{-k}(\Pi^-) \cap \Pi^+.
 $$
We now consider, for $k_0$ large enough, the corresponding return maps
\begin{equation}\label{e.return}
\ffR_{\mathbf{t}} \colon \Sigma^+_{\mathbf{t}}
\eqdef
\bigcup_{k\geqslant k_0}\Pi^+_{k,  \mathbf{t}} \to \Pi^+,\quad
\ffR_{\mathbf{t}}|_{\Pi^+_{k,  \mathbf{t}}}=\ffR_{k,\mathbf{t}} \eqdef
\ffT_{1,\mathbf{t}} \circ \ffT_{0,\mathbf{t}}^k.
\end{equation}
A saddle of $f_{\mathbf{t}}$ is said of {\em{single round type}} if it is a fixed point of
$\ffR_{\mathbf{t}}$.

\begin{remark}[The codimension one case]\label{r.variedades}
Recall the definition of
$\widetilde{\Pi}^{+}_{k}=\Pi^+_{j_*,k}$ in~\eqref{e.defPik2} and~\eqref{e.choicesdelta}.
Observe that the  $s$-boundary of $\widetilde{\Pi}^{+}_{k}$ is not
necessarily contained in the stable manifold of $P$ for
$f_{\mathbf{t}}$ when $\bf{t}\not=0$. However,
the stable manifold of $P$ on $\Pi^+$ varies smoothly with
$\mathbf{t}$, for $|\mathbf{t}|$ small enough. Then we can define
continuations of $\widetilde{\Pi}^{+}_{k, \bft}\subset \Pi^+_{k, \bft}$ for
small $|\mathbf{t}|$ in such a way  the $s$-boundaries
are still contained $W^s(P,f_{\mathbf{t}})$
and the
estimates for the $c$- and $u$-diameters in~\eqref{cl.inextremis}
hold.
\end{remark}

\subsection{Index variation}\label{ss.indexvariation}
Our next  goal  is to prove the following proposition, which is essentially borrowed from
\cite{Rom:95, GonShiTur:08}. For the next proposition see the cone fields  $\mathcal{C}^{cu}$ in Proposition~\ref{prop:cone} and recall the definition of the leading Jacobian in  \eqref{e.J}.

\begin{prop} \label{p.bis.eracl.2}
Consider $f\in \mathrm{Diff}^r(M^{m+n})$, $r,m\geqslant 2$ and $n\geqslant 1$, with a simple saddle $P$ of $u$-index $n$
 such that
 $J_P=J_P(f)>1.$
Assume that $P$ has  a homoclinic tangency
satisfying conditions {\bf{(C0)}} -- {\bf{(C5)}}.

Let
$(f_{\mathbf{t}})$ be the unfolding family associated to $f$ in
Section~\ref{ss.unfoldingreturn}. Then there is a sequence of open
sets of  parameters $\Delta_k\to \mathbf{0}$ such that for every
$\bft\in \Delta_k$ there are a contracting locally normally
hyperbolic $C^r$ submanifold $\mathcal{M}_{k,\bft}$ of dimension
$n+\mathrm{d}_{\mathrm{e}}$ tangent to the cone field $\mathcal{C}^{cu}$
satisfying:
\begin{itemize}
\item[(i)]
If $\mathrm{d}_{\mathrm{e}}=1$:
 there is
a  saddle $R_{k,\mathbf{t}}\in\mathcal{M}_{k,\mathbf{t}}$ of
$u$-index $n+1$.
\item[(ii)]
 If $\mathrm{d}_{\mathrm{e}}=2$:
\begin{itemize}
\item[$\bullet$]
 when $P$ is of type
$(1,2)$ and $\lambda\gamma>1$
 there are saddles $R_{k,\mathbf{t}},\,S_{k,\mathbf{t}}\in\mathcal{M}_{k,\mathbf{t}}$ of
$u$-indices $n-1$ and $n+1$, respectively,
\item[$\bullet$]
when  $P$ is of type
$(2,1)$ or
$(2,2)$ and $\lambda^2\gamma>1$
there are
saddles
$R_{k,\mathbf{t}}, \, S_{k,\mathbf{t}}
\in\mathcal{M}_{k,\mathbf{t}}$ of $u$-indices $n+1$ and $n+2$,
respectively,
\end{itemize}
\item[(iii)]
 If $\mathrm{d}_{\mathrm{e}}=3$:
  there are
saddles $Q_{k,\mathbf{t}},\, R_{k,\mathbf{t}},\, S_{k,\mathbf{t}}
\in\mathcal{M}_{k,\mathbf{t}}$ of $u$-indices $n-1$, $n+1$ and $n+2$, respectively.
\end{itemize}
Moreover, these saddles are of single round type and have
local unstable manifolds
contained in $\mathcal{M}_{k,\mathbf{t}}$.
\end{prop}

\begin{remark} \label{r.bis.eracl.2}
The submanifold $\mathcal{M}_{k,\mathbf{t}}$ is contained in the set
$\Pi^+_{k, \bft}$
 in~\eqref{e.defPik} for $k$ large enough and
$\bft \in \Delta_k$.
In the codimension one case, one gets that for
 $\mathcal{M}_{k,\bft} \subset
\widetilde{\Pi}^{+}_{k,\mathbf{t}}$ for every
large $k$ and $\mathbf{t} \in \Delta_k$.
\end{remark}

\begin{proof}
By~\cite[Theorem.~3, Equation~(1.2) and Lemma~2]{GonShiTur:08}
applied to
the unfolding family $(f^{-1}_{\bft})$ (see also~\cite[Theorems~A and
C]{Rom:95}), for sufficiently large $k$, there is a set $\Delta_k$
of parameters $\bft$, with $\Delta_k \to \mathbf{0}$ as $k \to
\infty$, such that the return maps $\ffR_{k,\bft}$ have a saddle
$R_{k,\bft}$ of $u$-index
 $n+1$. Besides,
 there are also saddles
$S_{k,\bft}$ of $u$-index $n+2$ (in cases (ii) and (iii)) and  $Q_{k,\bft}$ of $u$-index $n-1$
  (in case (iii)).
 These saddles are of single round type.

 Moreover, for the parameters $\bft\in \Delta_k$, there is also a
 contracting normally  hyperbolic $C^r$ manifold
$\mathcal{M}_{k,\bft}\subset  \Pi^+_{k,\bft}$
 (resp.~$\widetilde \Pi^+_{k, \mathbf{t}}$)
 which
is locally $f_{\bft}$ invariant, tangent to the cone field
$\mathcal{C}^{cu}$, and contains local unstable manifolds of
$R_{k,\bft}$,  $S_{k,\bft}$, $Q_{k,\bft}$ according to the cases. The dimension of
$\mathcal{M}_{k,\bft}$ is $n+ \mathrm{d}_{\mathrm{e}}$. This completes the proof.
\end{proof}

In the codimension one case we can obtain additional conclusions,
compare with  \cite[Lemmas 11 and 12]{LiTur:20}.

\begin{lem} \label{l.sebastopol}
Under the assumption of Proposition~\ref{p.bis.eracl.2}, if the saddle $P$ has $u$-index one
and it is
biaccumulated by its transverse homoclinic points,
then for every large $k$
and every $\bft\in \Delta_k$ it holds:
\begin{itemize}
\item
$W^u (R_{k,\bft},
f_{\mathbf{t}})\pitchfork
\partial^s\widetilde  \Pi^+_{k,\mathbf{t}} \ne \emptyset$,
in cases (i) and (ii)  of Proposition~\ref{p.bis.eracl.2},
\item
 $W^u (S_{k,\bft},  f_{\mathbf{t}})\pitchfork
\partial^s\widetilde \Pi^+_{k,\mathbf{t}} \ne \emptyset$,
in case (ii)  of Proposition~\ref{p.bis.eracl.2}.
\end{itemize}
In particular, as $\partial^s\widetilde \Pi^+_{k, \mathbf{t}}  \subset W^s(P,
 f_{\mathbf{t}})$,
it holds
$$
W^u (R_{k,\mathbf{t}},  f_{\mathbf{t}})\pitchfork W^s (P,
f_{\mathbf{t}})\ne \emptyset \quad \text{and} \quad W^u
(S_{k,\mathbf{t}},  f_{\mathbf{t}})\pitchfork W^s (P,
f_{\mathbf{t}})\ne \emptyset.
$$
\end{lem}

\begin{proof}
Note that $n=n_u=1$.
Fix large $k$ and $\bft \in \Delta_k$. Let $\mathcal{S}$
 be an $(m_s+1)$-dimensional submanifold contained in  $\widetilde \Pi^+_{k,\mathbf{t}}$ tangent
to the cone field $\mathcal{C}^{cu}$. This manifold is transverse to the
foliation $\widetilde{\mathcal{F}}^{ss}$ in
Remark~\ref{r.cgd}
 and the angle between them is uniformly bounded.
 Recall the linearising neighbourhood $W$ of $f$ and the linearising coordinates
 $(u,x,y)$, see~\eqref{e.linearisingcoor}.
 Consider the projection
  $\pi^{ss}\colon W \to \{ u=0 \}\subset W$
along the leaves of $\widetilde{\mathcal{F}}^{ss}$. By the $C^2$
regularity of $f$, the foliation $\widetilde{\mathcal{F}}^{ss}$ is
absolutely continuous.
 Thus
there are constants $c_1,c_2>0$ (independent of the submanifold
$\mathcal{S}$) such that
\begin{equation}\label{e.ok}
c_1\, \mathrm{vol}(\mathcal{S})
<\mathrm{vol}(\pi^{ss}(\mathcal{S})) < c_2 \,
\mathrm{vol}(\mathcal{S})
\end{equation}
where $\mathrm{vol}(\cdot)$ denotes $(m_s+1)$-volume.

Recall that the unfolding family $(f_{\bft})$ preserves the leading Jacobian of $P$.
Using Lemma~\ref{lema1} and Remark~\ref{r.uniform}, we get a constant $L>0$
 such that for every large
$k\geqslant 1$ it holds
\begin{equation}\label{e.ko}
\mathrm{vol}(\ffR_{k,\bft}(\mathcal{S})) > L\, J_{P}^k \,
\mathrm{vol}(\mathcal{S}).
\end{equation}
Thus, from~\eqref{e.ok} and~\eqref{e.ko}, it hold that
for
every $k$ large enough
\begin{equation}
\mathrm{vol}(\pi^{ss} (\ffR_{k,\bft}(\mathcal{S})))>c_1\,
c_2^{-1}\, L\, J_{P}^k\, \mathrm{vol}(\pi^{ss} (\mathcal{S}))=
\varrho_k \, \mathrm{vol}(\pi^{ss} (\mathcal{S}))
\end{equation}
where
$
\varrho_k \eqdef  c_1\, c_2^{-1}\, L\, J_{P}^k>1.
$
As $ J_{P}>1$ we get that $\varrho_k>1$ for every large $k$.
Arguing inductively, if $\ffR_{k,\bft}^{i}(\mathcal{S})\subset
\widetilde\Pi^+_{k,\bft}$ for every  $i=0,\dots,n-1$
 we get
\begin{equation}\label{e.5}
\mathrm{vol} (\pi^{ss}(\ffR^n_{k,\bft}(\cS))) > \varrho_k^n\,
\mathrm{vol}(\pi^{ss}(\mathcal{\cS})).
\end{equation}

We now focus on the saddles $R_{k,\bft}$. We have the following
facts about the sets $\widetilde \Pi_{k,\bft}^+$. First recall
Claim~\ref{cl.inextremis} about the ``unstable size" of the
$\widetilde\Pi_{k}^+$ and the definitions of $\partial^u
\widetilde \Pi_{k}^+$ and $\diam_{c}$, in~\eqref{e.bordes}
and~\eqref{e.diameter}, respectively.

\begin{remark}\label{r.intersection}. Let $\rho>0$ be the  constant in
Claim~\ref{cl.inextremis}.
Adjusting the ``unstable sides'' of  $\widetilde \Pi^+_{k,\bft}$, we
can assume that the saddle $R_{k,\bft}$ is
$(\frac{\rho}{10})$-\emph{centered} for every $k$ sufficiently
large:  for every submanifold
$\mathcal{S}$ of dimension $m_s+1$ tangent to $\mathcal{C}^{cu}$,
 containing $R_{k,\bft}$, and intersecting
 $\partial^u \widetilde \Pi_{k,\bft}^+$
 it holds
  $$
\mathrm{diam}_c (\mathcal{S}) > \frac{\rho}{10}.
$$
\end{remark}

The lemma (for $R_{k,\bft}$) follows immediately
from the next claim:

\begin{claim}\label{cl.hoy}
Consider a small disk $W^u\subset W^u(R_{k,\bft}, f_{\bft}) \cap
\mathcal{M}_{k,\bft}$ of dimension $(m_s+1)$ centered at
$R_{k,\bft}$.
Then there is $m_0$ such that
$\ffR_{k,\bft}^{m_0} (W^u)\pitchfork \partial^s \widetilde
\Pi^+_{k,\bft}\not = \emptyset.$
\end{claim}

\begin{proof}
By Equation  \eqref{e.5}, there is a first $n_0$ such that
$\ffR_{k,\bft}^{n_0+1} (W^u)\cap \partial
\widetilde\Pi^+_{k,\bft}\not=\emptyset$. If this intersection occurs in
$\partial^s \widetilde \Pi^+_{k,\bft}$ taking $m_0=n_0+1$ we are done.
Otherwise, the intersection occurs in $\partial^u \widetilde
\Pi^+_{k,\bft}$. We see that  this gives a contradiction.
Define
$$
W_0\eqdef C(R_{k,\bft}, \ffR_{k,\bft}^{n_0} (W^u)\cap \widetilde
\Pi^+_{k,\bft}), \quad W_1\eqdef \ffR_{k,\bft} (W_0),
$$
where
$C(A, \Upsilon)$ denotes the connected component of the set $\Upsilon$ containing the
point $A\in \Upsilon$.
Since $W_1$ intersects
$\partial^u \widetilde \Pi^+_{k,\bft}$,
$R_{k,\bft}\in W_1$ (recall that the saddle is single round),
 and $R_{k,\bft}$ is
$(\frac{\rho}{10})$-centered, Remark~\ref{r.intersection} implies
that
$$
\diam_{c} (W_1)> \frac{\rho}{10}.
$$
On the other hand, since $W_0$ is contained in $\widetilde
\Pi^+_{k,\bft}$, by item (b) in Claim~\ref{cl.inextremis} and
recalling~\eqref{e.choicesdelta}, it follows
$$
\diam_{u} (W_0)< \gamma^{-k}\, \delta_{j_*}.
$$
Using Lemma~\ref{l.fromRussiawithconstants} and Remark~\ref{r.uniform}, the previous
inequalities, and the choice of quantifiers in   Remark~\ref{r.choicesdelta}
we get
$$
\frac{\rho}{10}<\diam_c (W_1) =\diam_c (\ffR_{k,\bft}(W_0) ) < K
\, \gamma^k \, \diam_{u} (W_0)< K\, \delta_{j*} < \frac{\rho}{10}.
$$
This contradiction implies the claim.
\end{proof}
The proof of Lemma~\ref{l.sebastopol} for the saddle $R_{k,\bft}$
is now complete. The proof of the lemma for the saddle  $S_{k,
\bft}$ follows similarly and will be omitted.
\end{proof}

\section{Proof of  Theorems~\ref{thm:main} and \ref{thm:main2}}
\label{s.proofofmaintheorems}

Consider $f\in \mathrm{Diff}^r(M^{m+n})$
with a homoclinic tangency associated to a  saddle $P$ of
$u$-index $n$ and a blender $\Gamma$ as in Theorems~\ref{thm:main} or \ref{thm:main2}.  This means that
\begin{itemize}
\item  $P$ is simple of type $(m_s,n_u)$, with $m_s, n_u\in \{1,2\}$,
\item $J_P(f) >1$, and
\item
$P$ is homoclinically related to a blender $\Gamma$.
\end{itemize}
In Theorem~\ref{thm:main} we have that $n=n_u=1$ and that $\Gamma$ is a $cs$-blender of central dimension
$m_s$. In Theorem~\ref{thm:main2} we have that $\Gamma$ is a double blender of central dimensions
$(m_s,n_u)$.

Recalling Terminology~\ref{notation1}, from now on we use the following nomenclature:

\begin{terminology}[$C^r$ perturbation] A $C^r$ perturbation of $f$
is a $C^\infty$ diffeomorphism $g$ that can be obtained
arbitrarily $C^r$ close to $f$.
\end{terminology}

\subsection{Proof of Theorem~\ref{thm:main}}
\label{ss.proofoftheoremA}
The main step to prove Theorem~\ref{thm:main} is the following proposition which summarises the previous constructions.

\begin{prop} \label{p.l.n0bis}
Consider $f$ with a saddle $P$ and a $cs$-blender $\Gamma$ as in Theorem~\ref{thm:main}. There is a
{$C^r$}
perturbation ${g}$ of $f$, $r\geqslant 1$, having a  saddle $Q$
such that
\begin{itemize}
\item[(a)]
$Q$ is simple  of type $(m_s,1)$ and $J_{Q} (g)>1$,
\item[(b)]
$Q$ is homoclinically related to the blender $\Gamma_{g}$ and
to $P_{g}$,
\item[(c)] $Q$ is biaccumulated by its transverse homoclinic points,
\item[(d)] $Q$ has a homoclinic tangency in the superposition
domain of $\Gamma_g$, and
 \item[(e)]
the superposition region of the blender $\Gamma_{{g}}$ contains a
local strong stable manifold $W^{ss}_{\mathrm{loc}}(Q,
{g})$.
 \end{itemize}
\end{prop}

\begin{proof}
The proof of the proposition follows from Lemma~\ref{l.nonreal} taking
$P$ and $P'=Q^\ast$ a distinctive point of  the blender $\Gamma$ (which are
homoclinically related to $P$) and a small  neighbourhood $V$ of
$Q^\ast$ contained in the superposition domain of $\Gamma$.
Let $g$ and $Q$ be the diffeomorphism and the saddle
provided by Lemma~\ref{l.nonreal}. Note that the saddle
$Q$ satisfies
(a)--(b) and its orbit intersects $V$ (thus, after replacing by some iterate, we can assume that
$Q\in V$).
Since we have $n=1$, Remark~\ref{r.longmanifolds}
implies the biaccumulation property in (c).
Moreover, by Remark~\ref{r.preservation}, the saddle  $Q$ can be taken with a homoclinic tangency, obtaining (d).
Finally, item (e) also holds
provided $V$ is sufficiently small (so that
$W^s_{\mathrm{loc}}(Q,g)$ and $W^s_{\mathrm{loc}}(Q^*_g,g)$ are close enough) and
the fact that the disks of the superposition domain form an open
set. This completes the proof of the proposition.
\end{proof}

We are now ready to conclude the othe proof of Theorem~\ref{thm:main}.
Without lost of generality, after a $C^r$ perturbation,
 we can assume that the diffeomorphism $g$, the saddle $Q$, and its homoclinic tangency provided by Proposition~\ref{p.l.n0bis}
satisfy
conditions
 {\textbf{(C0)}}--{\textbf{(C4)}}.
 Using Proposition~\ref{p.bis.eracl.2},  we embed $g$ into an unfolding family
 $(g_{\mathbf{t}})$ (with one parameter if $m_s=1$ and two
parameters if $m_s=2$), getting open sets of parameters
$\Delta_k\to \mathbf{0}$, submanifolds $\mathcal{M}^k_{\bft}$, and
saddles $R^k_{\mathbf{t}}$ of $u$-index two and $S^k_{\mathbf{t}}$ (if $m_s=2$) of
$u$-index three,
$\bft\in \Delta_k$, as in
Proposition~\ref{p.bis.eracl.2} and Remark~\ref{r.bis.eracl.2}. Fix $k$ large enough (so Lemma~\ref{l.sebastopol} is
satisfied) and $\bft \in \Delta_k$. We will show that there are $C^r$
robust heterodimensional cycles between each saddle $R^k_{\bft}$
and $S^k_{\bft}$  and the blender $\Gamma_{\bft}=\Gamma_{g_{\bft}}$, see
Lemmas~\ref{l.inter1} and ~\ref{l.inter2} below.

\begin{lem}[Robust transverse intersections]
\label{l.inter1}
Let $\bft\in \Delta_k$. The unstable manifolds
$W^u(R^{k}_{\mathbf{t}},  g_{\mathbf{t}})$ and
$W^u(S^{k}_{\mathbf{t}},  g_{\mathbf{t}})$ (if $m_s=2$) transversely intersect
the stable set
$W^s(\Gamma_{\mathbf{t}},  g_{\mathbf{t}})$. Therefore these intersections
are $C^r$ robust.
\end{lem}

\begin{proof}
As the
$u$-index of $Q$ is one ($n=n_u=1$) it   is either of type $(1,1)$ or $(2,1)$.
First,  if $Q$ is of type $(1,1)$ then,  by Lemma~\ref{l.sebastopol}, for every
$\bft\in \Delta_k$ it holds
$$
W^u (R^{k}_{\mathbf{t}},  g_{\mathbf{t}})\pitchfork W^s
(Q_{{\mathbf{t}}},  g_{\mathbf{t}}) \ne \emptyset.
$$
Second, if $Q$ is of type $(2,1)$,  again by Lemma~\ref{l.sebastopol},  then for every $\bft\in \Delta_k$
it holds
  $$
  W^u(R^{k}_{\mathbf{t}},  g_{\mathbf{t}})\pitchfork
  W^s(Q _{{\mathbf{t}}},  g_{\mathbf{t}})\ne\emptyset
  \quad \mbox{and} \quad
 W^u(S^{k}_{\mathbf{t}},  g_{\mathbf{t}})
 \pitchfork W^s(Q _{{\mathbf{t}}},  g_{\mathbf{t}})\ne\emptyset.
 $$
In both cases, as $Q _{{\mathbf{t}}}$ and the  blender
$\Gamma_{{\mathbf{t}}}$ are homoclinically related, we get the
corresponding intersections between the  unstable manifolds of the
saddles $R^k_{{\mathbf{t}}}$ and $S^k_{{\mathbf{t}}}$ and the
stable set of $\Gamma_{\mathbf{t}}$.
\end{proof}

\begin{lem}[Robust quasi-transverse intersections]
\label{l.inter2}
Let $\bft\in \Delta_k$.
There is a $C^r$ neighbourhood $\mathcal{U}$ of $g_{\mathbf{t}}$
such that for every $h\in \mathcal{U}$  the  manifolds
and $W^{ss}(R^k_h,h)$ and
$W^{s}(S^k_h,h)$
 (if $m_s=2$)
contain disks in the superposition region of the blender
$\Gamma_h$. As a consequence, for every $h\in \mathcal{U}$ it holds
$$
W^{s}(S^k_h,h) \cap W^u_{\mathrm{loc}} (\Gamma_h,h)
\ne \emptyset \quad \mbox{and} \quad
W^{ss}(R^k_h,h) \cap W^u_{\mathrm{loc}}(\Gamma_h,h)
\ne \emptyset.
$$
\end{lem}

\begin{proof}
By Remark~\ref{r.setLambda}, the set $\Lambda_g$ consisting of the
orbits of
 $Q$ and the tangency point $Y$ is partially hyperbolic with a splitting of the form
 $E^{ss} \oplus E^{cu}$, where $\dim (E^{ss})=m- m_s$.
 Remark~\ref{r.cgd} provides the foliation $\widetilde{\mathcal{F}}^{ss}_g$
defined on
 a neighbourhood of $\Lambda_g$ and  tangent to the cone-field $\mathcal{C}^{ss}$.
 The leaves of this foliation have uniform size
(depending only  on  the neighbourhood of $g$)
 and
restricted to the set $\Lambda_g$ is invariant. This, in
particular, implies that
  $\widetilde{\mathcal{F}}^{ss}_{\mathrm{loc}}(Q)$
   is a local strong stable manifold
  $W^{ss}_{\mathrm{loc}}(Q,g)$ of $Q$.
Moreover,  by item (e) of  Proposition~\ref{p.l.n0bis}, the set
$W^{ss}_{\mathrm{loc}}(Q,g)$ is a disk in the region of
superposition $ \mathscr{D}^{ss}$ of the blender $\Gamma_g$.

Note that for $k$ large enough the orbit of $R^k_\bft$ can be
chosen with iterates arbitrarily close to the orbit of $Y$,  in particular, close to $Q$.
Thus, after replacing by some iterate, we can assume that $R^k_\bft$ is close enough to $Q$ so that
 its local strong stable manifold
$W^{ss}_{\mathrm{loc}}(R_\bft^k,g_\bft) =
\widetilde{\mathcal{F}}^{ss}_{\mathrm{loc}}(R^k_\bft)$  is close
to $W^{ss}_{\mathrm{loc}}(Q_g,g)$.  Since the superposition domain
 $B^{ss}$ of the blender $\Gamma_g$ is an
open set and the (local) strong stable manifold of $R_{\bft}^k$ varies $C^1$
continuously in a neighbourhood of $g_\bft$, we have that (provided that $\mathcal{U}$ is small)
$W^{ss}_{\mathrm{loc}}(R_h^k,h)$ is also a disk in $\mathscr{D}^{ss}$.
By item  (1) in
Remark~\ref{r.blender} about
continuations of
blenders, the superposition domain  ${B}^{ss}$
is also a
superposition domain of the continuation $\Gamma_h$ of $\Gamma_{\bft}$.
The robust intersection property of blenders (item (c) (ii) in Definition~\ref{def:blender})
implies that
$W^{ss}_{\mathrm{loc}}(R^k_h,h) \cap W^u_{\mathrm{loc}}(\Gamma_h,h)
\ne \emptyset$, proving the lemma for the saddle $R_h^k$. The proof
for the saddle $S_\bft^k$ is analogous and omitted.
\end{proof}

Note that if $m_2=2$, we get robust cycles of
coindex one (associated to $R^k_h$ and $\Gamma_h$) and two
 (associated to $S^k_h$ and $\Gamma_h$).
This ends the proof of Theorem~\ref{thm:main}. \hfill $\qed$

\subsection{Proof of Theorem~\ref{thm:main2}}\label{ss.proofoftheoremB}

The proof of Theorem~\ref{thm:main2} is similar to the one of
 Theorem~\ref{thm:main}.
The robust quasi-transverse intersections are obtained exactly as in Theorem~\ref{thm:main} (see Lemma~\ref{l.inter22}). However, since the biaccumulation property is not available in the general case, the proof of the existence of transverse intersections must be different.  For that, we use the geometrical properties of double blenders (see Lemma~\ref{l.inter11}).

The main step to prove Theorem~\ref{thm:main2} is the following proposition
(analogous to Proposition~\ref{p.l.n0bis}) summarising the constructions in previous sections when $n>1$.

\begin{prop}
\label{p.l.n0bis2}
Consider $f$ with a saddle $P$ and a double blender $\Gamma$ as in Theorem~\ref{thm:main2}.
There is a
{$C^r$}
perturbation ${g}$ of $f$, $r\geqslant 1$, having a  saddle $Q$
such that
\begin{itemize}
\item[(a)]
$Q$ is simple  of type $(m_s,n_u)$ and $J_{Q} (g)>1$,
\item[(b)]
$Q$ is homoclinically related to the double blender $\Gamma_{g}$ and
to $P_{g}$,
\item[(c)] $Q$ has a homoclinic tangency in the $ss$-superposition domain
of $\Gamma_g$, and
 \item[(d)]
the $uu$-superposition region of the double blender $\Gamma_{{g}}$ contains a
local strong unstable manifold of some point of the orbit of $Q$.
 \end{itemize}
\end{prop}

\begin{proof}
Recall that the double blender $\Gamma$ has a pair of
distinctive saddles $Q^\ast_{cs}$ and $Q^\ast_{cu}$ and
$ss$- and $uu$-superposition regions
($\mathscr{D}^{ss}$ and $\mathscr{D}^{uu}$)
 and  domains ($B^{ss}$ and $B^{uu}$), see Definition~\ref{def:blender}.

Let $U$ and $V$ be neighbourhoods of $Q^\ast_{cs}$ and $Q^\ast_{cu}$ such that $U\subset B^{ss}$ and $V\subset B^{uu}$. By hypothesis, $Q^\ast_{cu}$  is homoclinically related to $Q^\ast_{cs}$ and $P$.
Having this in mind, a slight variation of Lemma~\ref{l.nonreal} applied
$P$ and $P'=Q^\ast_{cs}$, we get a diffeomorphism $g$ with a saddle $Q$ satisfying
(a)--(b)
such that $Q\in U$ and  has some iterate $Q'\in V$.
Taking $V$ sufficiently small, the continuous dependence of the
 local invariant  manifolds (with respect to the point and the diffeomorphism) implies that
$W^{ss}(Q,g)$  contains a disk in $\mathscr{D}^{ss}$
and $W^{uu}(Q',g)$ contains a disk in $\mathscr{D}^{uu}$.
Finally, by Remark~\ref{r.preservation}, the saddle  $Q$ can be taken with a homoclinic tangency, obtaining (c)-(d).
This completes the proof of the proposition.
\end{proof}

We are now ready to complete the proof of Theorem B.
Without loss of generality, after a $C^r$ perturbation,
 we can assume that the saddle $Q$ and its homoclinic tangency provided by Proposition~\ref{p.l.n0bis2}
satisfy
conditions
 {\textbf{(C0)}}--{\textbf{(C5)}}.
 As in the proof of Theorem~\ref{thm:main}, we
 apply Proposition~\ref{p.bis.eracl.2} to $Q$ and $g$ to get an unfolding family
 $(g_{\mathbf{t}})$ and sets of parameters $\Delta_k\to \mathbf{0}$
such that for each $\mathbf{t}\in \Delta_k$
 \begin{itemize}
 \item [i)]
 if $m_s=1$ then there exists a saddle $R^k_{\mathbf{t}}$ of $u$-index $n+1$, and
 \item  [ii)]
 if $m_s=2$ then there exist two saddles $R^k_{\mathbf{t}}$ and $S^k_{\mathbf{t}}$ of $u$-indices $n+1$ and $n+2$, respectively.
 \end{itemize}

Now we are going to verify the existence of robust cycles stated in the theorem.

\begin{lem}[Robust transverse intersections]
\label{l.inter11}
Let $\bft\in \Delta_k$. The unstable manifolds
$W^u(R^{k}_{\mathbf{t}},  g_{\mathbf{t}})$ and
$W^u(S^{k}_{\mathbf{t}},  g_{\mathbf{t}})$ (if $m_s=2$) transversely intersect
the stable set
$W^s(\Gamma_{\mathbf{t}},  g_{\mathbf{t}})$. Therefore these intersections
are $C^r$ robust.
\end{lem}

\begin{proof}
By construction, the orbits of
$R^k_{\mathbf{t}}$ and $S^k_{\mathbf{t}}$ follow  the orbit of $Q$, thus they intersect the open set $V$.
Thus, taking $V$ small enough, there are points
$\widehat{R}^k_{\mathbf{t}}\in \mathcal{O}(R^k_{\mathbf{t}})$ and $\widehat{S}^k_{\mathbf{t}}\in \mathcal{O}(S^k_{\mathbf{t}})$ such that
both
strong unstable manifolds $W^{uu}(\widehat{R}^k_{\mathbf{t}},g_{\mathbf{t}})$ and $W^{uu}(\widehat{S}^k_{\mathbf{t}},g_{\mathbf{t}})$ contain a disk in $\mathscr{D}^{uu}$. The
robust intersection property of the $cu$-blender $\Gamma_\mathbf{t}$,
see (c)(ii) in Definition~\ref{def:blender}, now implies that the intersections
$$
W^{u}(\widehat{R}^k_{\mathbf{t}},g_{\mathbf{t}})\cap
W^{s}(\Gamma_\mathbf{t},g_{\mathbf{t}})\neq \emptyset
\quad
\mbox{and}
\quad
W^{u}(\widehat{S}^k_{\mathbf{t}},g_{\mathbf{t}})
\cap
W^{s}(\Gamma_\mathbf{t},g_{\mathbf{t}})
\neq \emptyset
$$
are $C^r$-robust. Finally, note that since the $u$-indices of $\Gamma_{\mathbf{t}}$,
$R^k_{\mathbf{t}}$, and $S^k_{\mathbf{t}}$ (when $m_s=2$) are
  $n$, $n+1$, and $n+2$, respectively, the intersections above can be done transverse after a perturbation.
\end{proof}

\begin{lem}[Robust quasi-transverse intersections]
\label{l.inter22}
Let $\bft\in \Delta_k$.
There is a $C^r$ neighbourhood $\mathcal{U}$ of $g_{\mathbf{t}}$
such that for every $h\in \mathcal{U}$  the  manifolds
and $W^{ss}(R^k_h,h)$ and
$W^{s}(S^k_h,h)$
 (if $m_s=2$)
contain disks in the superposition region of the blender
$\Gamma_h$. As a consequence, for every $h\in \mathcal{U}$ it holds
$$
W^{s}(S^k_h,h) \cap W^u_{\mathrm{loc}} (\Gamma_h,h)
\ne \emptyset \quad \mbox{and} \quad
W^{ss}(R^k_h,h) \cap W^u_{\mathrm{loc}}(\Gamma_h,h)
\ne \emptyset.
$$
\end{lem}

\begin{proof}
As in the proof of  Theorem A, the
strong stable manifolds $W^{ss}(R^k_{\mathbf{t}},g_{\mathbf{t}})$ (in both cases i) and ii)) and $W^{ss}(S^k_{\mathbf{t}},g_{\mathbf{t}})$ contain a disk in $\mathscr{D}^{ss}$. The
robust intersection property of the $cs$-blender $\Gamma_\mathbf{t}$
provides the  $C^r$-robust
 intersections
$$
W^{s}(R^k_{\mathbf{t}},g_{\mathbf{t}})\cap
W^{u}(\Gamma_\mathbf{t},g_{\mathbf{t}})\neq \emptyset
\quad
\mbox{and}
\quad
W^{s}(S^k_{\mathbf{t}},g_{\mathbf{t}})
\cap
W^{u}(\Gamma_\mathbf{t},g_{\mathbf{t}})
\neq \emptyset,
$$
proving the lemma.
\end{proof}

Summarising, the sets
$\mathcal{O}(R^k_{\mathbf{t}})$ and $\Gamma_\mathbf{t}$ form a $C^1$ robust heterodimensional cycles of coindex one and the sets $\mathcal{O}(S^k_{\mathbf{t}})$ and $\Gamma_\mathbf{t}$ form a $C^r$-robust heterodimensional cycles of coindex two.
The proof of Theorem B is now complete. \hfill $\qed$

\begin{remark}
When the effective dimension  is three, we can also obtain robust cycles of coindex one associated
to the continuation of the double blender  and saddles of $u$-index $n-1$.
\end{remark}

\section{Simultaneity of robust cycles: Simon-Asaoka's examples}
\label{s.simonasaokatan}
In this section, we use
Theorem~\ref{thm:main} to prove
Corollary~\ref{c.p.asaoka-simon},
claiming that
 the diffeomorphisms with robust
homoclinic tangencies
in~\cite{Sim:72b,Asa:08} also display $C^1$ robust
heterodimensional cycles. We start by introducing the
$cs$-blenders of
Plykin type.

\subsection{$cs$-blenders of Plykin type}
\label{ss.blendersplykin}
Consider a two-dimensional diffeomorphism $h$ with
a {\em{Plykin repeller}} $\Sigma$
inside  the unitary disk
$\mathbb{D}\subset \mathbb{R}^2$. For a detailed description of this repeller, see \cite[Chapter 8.6]{Rob:99}, for instance.
The key property of the  repeller $\Sigma$ is that the disk $\mathbb{D}$
is foliated by local unstable manifolds
$W^u_{\mathrm{loc}}(p, h)$
of points $p \in \Sigma$.  The repeller $\Sigma$ has a fixed point,
say the
origin $(0,0)$,
such that
\begin{equation}
\label{e.contracting}
Dh(0,0)
=\begin{pmatrix}\lambda & 0\\0 & \sigma \end{pmatrix},
\quad 0<\lambda<1<\sigma.
\end{equation}

Consider $U\eqdef (-1,1)^{n-2} \times \mathbb{D}$
and a local diffeomorphisms of the form
$$
\varphi \colon U \to  \mathbb{R}^{n-2}  \times \mathbb{D}
\quad  \varphi (x,y,z) = (\gamma x, h(y,z) )
$$
where
$$
 0 < \gamma <\inf_{(y,z)\in \mathbb{D}}\Vert Dh(y,z)\Vert.
 $$
Thus
$\Gamma\eqdef\
 \{0^{n-2}\} \times \Sigma$
 is a hyperbolic set of $\varphi$  and $P=(0^{n})\in \Gamma$ is a fixed point of  $\varphi$ of
 $u$-index one.
 Note also that $\Gamma$ is the maximal invariant set of $\varphi$ in $U$.
 See Figure~\ref{fig:Ply}.

 Given now any compact manifold $M^n$, $n\geqslant 3$,
 we can embed the local diffeomorphism $\varphi$ into a diffeomorphism $f_\varphi=f \in
 \mathrm{Diff}^r (M^n)$. With a slight abuse of notation, we denote by $\Gamma$ and $P$ the
 corresponding
 hyperbolic set and saddle of $f$.

\begin{figure}
\centering
\begin{overpic}[scale=0.04,
]{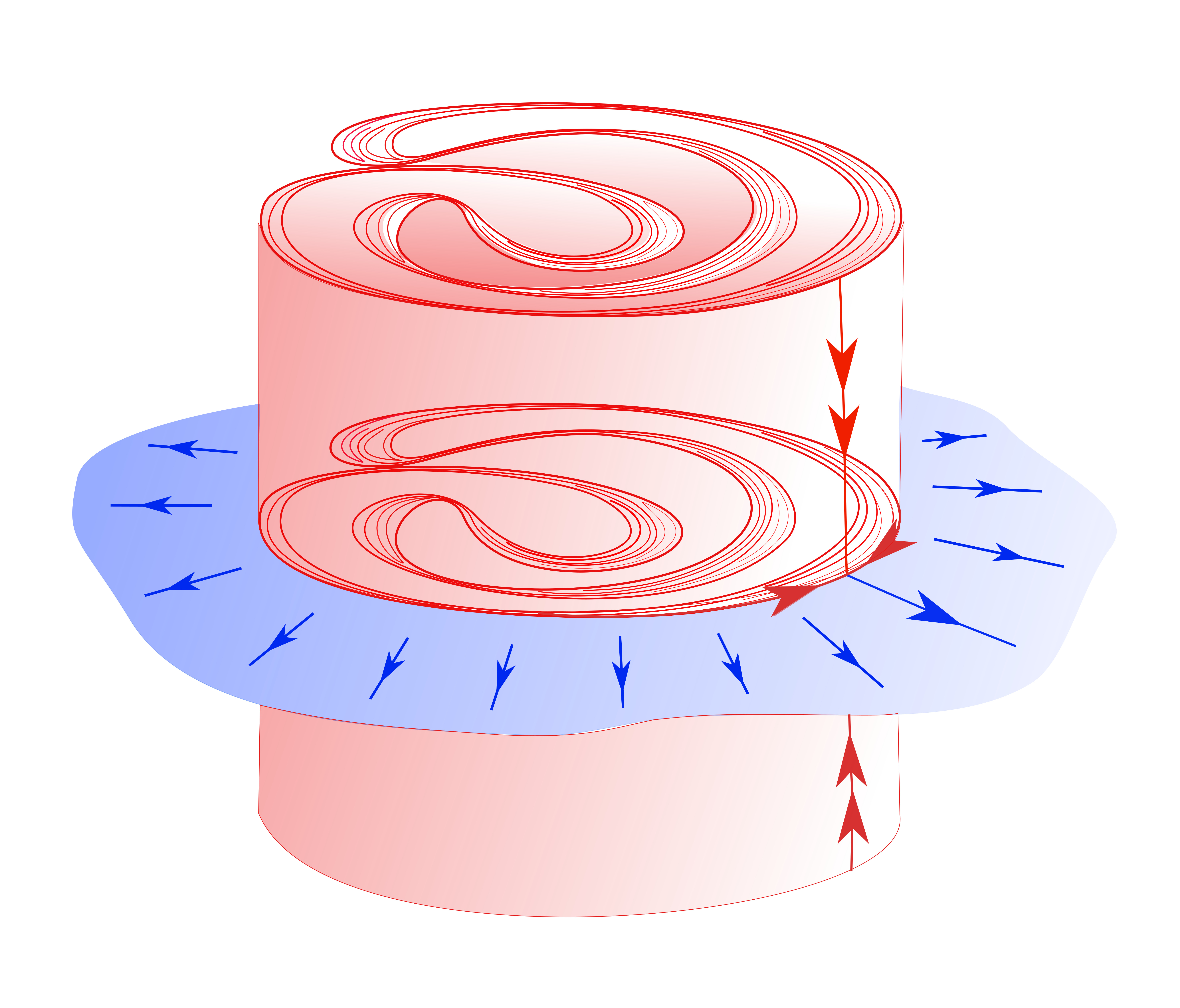}
  \put(192,87){\Large {$P$}}
 \end{overpic}
 \caption{\label{fig:Ply}A $cs$-blender of Plykin type }
\end{figure}

 \begin{remark} \label{r.foliated}
 By construction,
 the surface $\mathcal{S}\eqdef  \{0^{n-2}\} \times \mathbb{D}$ is normally hyperbolic.
Thus there is a $C^r$ neighbourhood $\mathscr{U}$ of $f$ such
 that for every $g\in \mathscr{U}$ there are well defined continuations $\Gamma_g$ and $\mathcal{S}_g$ of $\Gamma$ and $\mathcal{S}$, respectively, with
 $\Gamma_g \subset \mathcal{S}_g$. As $\mathbb{D} \subset \mathbb{R}^2$ is foliated by the (local) unstable manifolds
$W^u_{\mathbb{D}, \mathrm{loc}}(p, h)$, $p \in \Sigma$, by normal
hyperbolicity of $\mathcal{S}_g$, see~\cite{HirPugShu:77}, we have that
$\mathcal{S}_g$ is also foliated by local unstable manifolds
of $\Gamma_g$.
\end{remark}

Next lemma is just a reformulation of the discussion in \cite[Section 4.2]{Bon:11}.

  \begin{lem}
  \label{l.plykinblender}
  The set $\Gamma$ is a $cs$-blender.
  \end{lem}
 We say the blender in Lemma~\ref{l.plykinblender} is of {\em{Plykin type.}}
   As a consequence of item (1) in Remark~\ref{r.blender}, for
  every $g$ sufficiently $C^1$ close to $f$,
the continuation $\Gamma_g$ of $\Gamma$ is also $cs$-blender (of Plykin type).

 \begin{proof}
Since having a $cs$-blender is a $C^1$ robust property, it is enough to prove the lemma for $f$.
Properties (a) of local maximality of $\Gamma$ and (b) partial hyperbolicity of $\Gamma$ follows considering
any neighbourhood
$V=(-\delta,\delta)^{n-2}\times \mathbb{D}$ (small $\delta>0$) of $\Gamma$
 and the (dominated) splitting $T_\Gamma M^n = E^{ss}\oplus E^{cs} \oplus E^{u}$, where
 $E^{ss}$ is $\mathbb{R}^{n-2}\times \{0^2\}$ and $T_\Sigma \mathbb{R}^2=E^{cs}\oplus E^{u}$ is the hyperbolic splitting of $\Sigma$.

To define the family of disks $\mathscr{D}^{ss}$
of the superposition region
 of $\Gamma$, consider
first a thin cone field
$\mathcal{C}^{ss}$  associated to the strong stable direction $E^{ss}$ of $\Gamma$.
 We extend $\mathcal{C}^{ss}$ to a small neighbourhood $W$ of $V$ (maybe after shrinking $\delta$) and, for  simplicity, denote this extension also by $\mathcal{C}^{ss}$. Recall that a $(n-2)$-dimensional disk
 $\Delta \subset V$  is tangent to the cone field $\mathcal{C}^{ss}$ if
$T_p \Delta \subset \mathcal{C}^{ss}(p)$ for every $p\in \Delta$.
The family  $\mathscr{D}^{ss}$
consists of all $(n-2)$-dimensional disks $\Delta$ containing in $W$ and
 tangent
to $\mathcal{C}^{ss}$
whose intersection with the boundary of $V$ has two connected components.
This defines an open family of
 disks.

Given any saddle $Q\in
\Gamma$, consider its local strong stable manifold  relative to $U$
$W^{ss}_{U, \mathrm{loc}}(Q,f)$.
This set contains a disk of $\mathscr{D}^{ss}$ containing $Q$ in
its interior. Thus, any periodic point of $\Gamma$ can be chosen as
a distinctive saddle of the blender, proving property (c-i) in the
definition of a $cs$-blender.

To get property (c-ii) about robust intersections, just observe
that
 $
 W^{u}_{U, \mathrm{loc}} (\Gamma,f) = \mathcal{S}
 $
 and that   for  every disk $D$ of  $\mathscr{D}^{ss}$ it holds $\mathcal{S} \pitchfork D\neq \emptyset$. The proof of the lemma is now complete.
  \end{proof}

 \subsection{Local deformations: generation of robust tangencies} \label{sss.robusttangencies}
   To generate a robust homoclinic tangency associated to the blender $\Gamma$ we argue as in \cite{Sim:72b,Asa:08}.
  Let us sketch the main arguments there. Recall that $P$ is the fixed point of $\Gamma$. Consider a small $(n-1)$-dimensional disk $\Delta\subset W^s (P,f)$ disjoint from
  the closure of  $V$, hence  disjoint from $\Gamma$. Through a smooth isotopy supported outside  the closure of $V$, we carry $\Delta$ close to $\Gamma$ in such a way that the following properties hold:
\begin{itemize}
\item[(i)]
$\mathcal{T}\eqdef \Delta\cap V$ is a  {\em{$ss$-tube}}:
$\mathcal{T}$ is a manifold diffeomorphic to $(-1,1)^{n-2} \times
\mathbb{S}^1$, foliated by disks  of $\mathscr{D}^{ss}$.
\item [(ii)] the transverse intersection between $\mathcal{T}$ and $\mathcal{S}$ is a simple closed curve $\alpha$,
\item [(iii)] the curve $\alpha$ (and thus the manifold $\mathcal{T}$) is tangent to $
   W^{\mathrm{u}}_{\mathrm{loc}}(P,f)$.
\end{itemize}
This  isotopy produces a new diffeomorphism whose restriction to $V$ coincides with $f$. Hence $\Gamma$ is a $cs$-blender of this new diffeomorphism. With a slight abuse of notation, we continue calling the resulting diffeomorphism by $f$.
We claim that for every $g$ sufficiently $C^r$ close  to $f$ the set $\Gamma_g$ has a homoclinic tangency and thus $\Gamma$ has a $C^r$ robust tangency.
 Indeed,  consider the continuations $\mathcal{T}_g$ of $\mathcal{T}$ and $\mathcal{S}_g$ of $\mathcal{S}$ for  $g$ close to $f$. Note that  $\mathcal{T}_g\subset W^s(P_g,g)$ is also a $ss$-tube  and    $\alpha_g=\mathcal{T}_g\pitchfork \mathcal{S}_g$ is  also a  simple closed  curve. On the other hand, as $\mathcal{S}_g=W^{u}_{U,\mathrm{loc}} (\Gamma_g,g)$,
 there is $P_0\in \Gamma_g$ such that $W^u_{\mathrm{loc}} (P_0,g)$ and $\alpha_g$ are tangent
(this is a consequence of Remark~\ref{r.foliated}).

We denote by $ \mathscr{T}$ the $C^r$ neighbourhood of $f$ obtained above.

\subsection{Proof of Corollary~\ref{c.p.asaoka-simon}}
The first item of the proposition was obtained above. The second item follows applying Theorem~\ref{thm:main} to $f$ and $P$.
Note  that $P$   is a simple saddle
of type $(1,1)$
with leading Jacobian $J_P (f) = \lambda\, \sigma>1$ and
 homoclinically related to $\Gamma$.  After shrinking $\mathscr{T}$, we can assume that these conditions hold for
 $P_g$ and $\Gamma_g$ for every $g\in \mathscr{T}$.
By the second part of Remark~\ref{r.tangenciesandcontinuations},
there is a dense subset
$\mathscr{D}$ of $\mathscr{T}$ of diffeomorphisms $g$ having a homoclinic tangency associated
to $P_g$. The diffeomorphisms in $\mathscr{D}$  satisfy the hypotheses of Theorem~\ref{thm:main}.
Thus for each $g\in \mathscr{D}$ there is an open set $\mathscr{C}_g$ whose closure contains $g$ consisting of diffeomorphisms with robust heterodimensional cycles as above.
Taking now
$$
\mathscr{C} \eqdef \bigcup_{g\in \mathscr{D}} \mathscr{C}_g.
$$
the corollary follows.
\hfill \qed

\section{Appendix:  Horseshoes with large entropy and blenders.\\
Proof of Theorem~\ref{mainthm-hyperbolic-measure}}
\label{s.appendix}

In this appendix, we explain how Theorem~\ref{mainthm-hyperbolic-measure} follows from
the results in \cite{AviCroWil:21}.

By the assumption, the hyperbolic measure $\mu$ is such that there is $\epsilon>0$ with
\begin{equation}
\label{e.byassumption}
 h_{\mu}(f) > -\log \mathrm{J}_{\mu}^{cs}(f)+ \frac{1}{2 r} \chi^{cs}_{\mu}(f) +
\epsilon.
\end{equation}
 Applying \cite[Theorem~B']{AviCroWil:21} to the
 hyperbolic measure  $\mu$, given any $\delta
>0$ we get  $C^1$ perturbation $h \in \mathrm{Diff}^r(M)$ of $f$
with an affine
horseshoe\footnote{A horseshoe $\Lambda$ is said
\emph{affine} if there are a neighbourhood $U$ of $\Lambda$ and
a chart $\varphi \colon U\to \mathbb{R}^{d}$ such that $\varphi\circ f
\circ \varphi^{-1}$ is locally affine in $U$. If the chart $\varphi$ can be chosen such that $D(\varphi\circ f
\circ \varphi^{-1})(x)$ is a matrix $A$ independent of $x$, we say that $\Lambda$ has \emph{constant linear part $A$}.}
$\Gamma_\delta$ whose hyperbolic splitting $E_\delta^s\oplus
E_\delta^u$
 satisfies
\begin{itemize}
\item
$\Gamma_\delta$ has a constant linear part  $A_\delta$ which is a
diagonal matrix with distinct real positive eigenvalues,
\item
The Lyapunov exponents of $\mu$ are $\delta$-close to the Lyapunov
exponents of $A_\delta$. In particular,
\begin{equation*}
\label{e.limitjacobians} \mathrm{Jac}_{E^s_\delta}(A_\delta) \to
\mathrm{J}^s_\mu \quad \mbox{and} \quad \chi^{cs}(A_\delta) \to
\chi^{cs}_\mu \quad \mbox{as $\delta \to 0$},
\end{equation*}
where $\mathrm{Jac}_{E_\delta^s}(A_\delta)\eqdef \det A_\delta|_{E_\delta^s}$ and
$\chi^{cs}(A_\delta)$ is the negative Lyapunov exponent of $A_\delta$ closest to
zero. From this convergence, taking $\delta>0$ small enough, we
have
\begin{equation}
\label{e.inparticular}
\begin{split}
-\log \mathrm{J}^{s}_{\mu}+ \frac{\epsilon - \delta}2  &> -\log
\mathrm{Jac}_{E_\delta^s}(A_\delta) \quad
\mbox{and}\\
 \frac{1}{2 r} \chi^{cs}_{\mu} + \frac{\epsilon - \delta}2
&>  \frac{1}{2 r} \chi^{cs}(A_\delta),
\end{split}
\end{equation}
\item
$\Gamma_\delta$ is $\delta$-close to the support of $\mu$ in Hausdorff
distance,
\item
$h_{\mathrm{top}}({\Gamma_\delta},{h_\delta})>h_{\mu}(f)-\delta$.
\end{itemize}

The last inequality implies that
\begin{equation}
\label{e.requeriment}
\begin{split}
h_{\mathrm{top}}({\Gamma_\delta},{h_\delta})&>h_{\mu}(f)-\delta \overset{\eqref{e.byassumption}}{\geqslant}
-\log \mathrm{J}^{s}_{\mu}+ \frac{1}{2 r} \chi^{cs}_{\mu} + \varepsilon - \delta  \\
& \overset{\eqref{e.inparticular}}{>} -\log \mathrm{Jac}_{E_\delta^s}(A_\delta)+ \frac{1}{2 r} \chi^{cs}(A_\delta).
\end{split}
\end{equation}

Observe that the inequality in \eqref{e.requeriment} allows us to
apply~\cite[Theorem~C]{AviCroWil:21} to the affine horseshoe $\Gamma_\delta$
of $h_\delta$, getting
 a perturbation $g_\delta$ of $h_\delta$ supported in a small
neighbourhood of $\Gamma_\delta$  such that the continuation $\Gamma_{g_\delta}$ of
$\Gamma_\delta$ is a $cs$-blender of central dimension $d_{cs}=m-1$ with a
dominated splitting whose bundles are one-dimensional.
Taking now a sequence
$(\delta_k)\to 0$, for each large $k$ we get diffeomorphisms $h_k=h_{\delta_k}$
and perturbations $g_k$ of $h_k$ with blenders $\Gamma_k$ as before.
These blenders satisfy
conditions \eqref{i.bl1}, \eqref{i.bl2}, and \eqref{i.bl3}
in the theorem.

To conclude the proof of the theorem, it remains to see that if the saddle $P$ is
 homoclinically related to $\mu$
then we can chose the perturbations $g_k$ such that the
continuations $P_k$ of $P$ and $\Gamma_k$ of $\Gamma$ are homoclinically
related. To see this, we need to review the steps in the proof of
\cite[Theorem~B']{AviCroWil:21}, which is a combination
of~\cite[Theorem~B]{AviCroWil:21} and Katok's approximation
theorem \cite{Kat:80} and its extensions in
\cite[Theorem.~3.3]{AviCroWil:21}
and~\cite[Theorem~2.12]{BuzCroOm:18}.

The first step is to apply Katok's approximation theorem to the
hyperbolic measure $\mu$ of $f$ to obtain a horseshoe
$\Lambda$ of $f$ such that:
\begin{enumerate}
\item[(a)] $\Lambda$ is homoclinically related to $\mu$,
\item[(b)]
$\Lambda$ is close to the support of $\mu$ in the Hausdorff
distance,
\item[(c)] the topological entropy of $\Lambda$ is close to $h_\mu(f)$,
\item[(d)] the $f$-invariant measures of $\Lambda$ are close to
$\mu$ in the weak* topology, and
\item[(e)] the Lyapunov exponents of any ergodic
measure of $\Lambda$
are close to the Lyapunov exponents of $\mu$.
\end{enumerate}
By hypothesis, the saddle $P$ is homoclinically related to $\mu$. Thus (a) and
the $\lambda$-lemma imply that
 $P$ and $\Lambda$ are homoclinically related.
We can assume that $P\not \in \Lambda$.
Otherwise, if $P\in \Lambda$, we extract a subhorseshoe
$\widetilde{\Lambda}$ of $\Lambda$ such that $P \not \in
\widetilde{\Lambda}$ and $h_{\mathrm{top}}(\widetilde \Lambda,f)$
is close to $h_{\mathrm{top}}(\Lambda,f)$. In particular,
 properties (a)--(e)   also hold for
$\tilde{\Lambda}$.

The second step in the proof of~\cite[Theorem~B']{AviCroWil:21} is to
apply~\cite[Theorem~B]{AviCroWil:21} to the horseshoe $\Lambda$
to get affine horseshoes with constant linear part. Since the perturbation in the latter result is done in a small
neighbourhood of $\Lambda$ and $P\not \in \Lambda$, we obtain a
$C^1$ perturbation $h \in \mathrm{Diff}^r(M)$ of $f$
with
$P_h=P$ and such that the continuation $\Lambda_h$ of $\Lambda$
contains an affine horseshoe $\Gamma_h$ with constant linear
part.  In particular, $P_h$ is
homoclinically related with~$\Gamma_h$.

Finally,  when $P$ has a homoclinic tangency, the constructions above can be done preserving that tangency. This concludes the proof of the theorem.

\section{Appendix: Dominated splittings}
\label{s.Sinovac} In this appendix, we review  the notion of a dominated splitting.

Given an invariant set $\Lambda$ of
$f \in \mathrm{Diff}^1(M)$, a splitting $T_\Lambda M= E\oplus F$ over
$\Lambda$ is called {\em{dominated}} if it is $Df$ invariant and
there is $\ell\in \mathbb{N}$ such that for every $x\in \Lambda$
and every pair of unit vectors $u \in E(x)$ and $v\in F(x)$ it
holds
$$
\frac{|| Df^\ell (x) (u)||}{|| Df^\ell (x) (v)||} < \frac{1}{2},
$$
here $E(x)$ and $F(x)$ are the bundles at $x$ and $|| \cdot||$ is
the norm. In this case, we say that $F$ {\em{dominates}} $E$.

A $Df$ invariant bundle $T_\Lambda M= E_1\oplus \cdots \oplus
E_k$ with $k$ bundles, $k \geqslant 3$, is {\em{dominated}} if the
bundles  $T_\Lambda M= (E_1\oplus \cdots \oplus E_i) \oplus (
E_{i+1} \oplus \cdots \oplus E_k)$ are dominated for every $i\in
\{1,\dots, k-1\}$.

The bundle $E$ over $\Lambda$ is {\em{uniformly contracting}} if
there are constants $C>0$ and $\kappa\in (0,1)$ such that
$$
|| Df^n (x) (u)|| \leqslant C \, \kappa^n || u||
$$
for every $x\in \Lambda$, $n \geqslant 0$, and every vector $u\in
E(x)$. Similarly, the bundle $F$ is   {\em{uniformly expanding}}
if it is uniformly contracting for $f^{-1}$. A dominated splitting
$E \oplus F$  over $\Lambda$ is {\em{partially hyperbolic}} if
either $E$ is uniformly contracting or $F$ is uniformly expanding.
When $E$ is uniformly contracting and $F$ is uniformly expanding
the splitting is {\em{hyperbolic}}. An invariant set $\Lambda$ is
({\em{partially}}) {\em{hyperbolic}}  if it
has a (partially) hyperbolic splitting.

The following proposition is a well-known result about dominated
splittings. We refer to \cite[Chapter B.2]{BonDiaVia:05} for
details (including the definitions of cone fields).

\begin{prop}
\label{prop:cone} Let  $f\in \mathrm{Diff}^1(M)$  and
$\Lambda\subset M$ be a compact $f$-invariant set with a
dominating splitting $ T_\Lambda M=E \oplus F.$ Then there are
neighbourhoods $U$ of $\Lambda$ in $M$ and $\mathscr{N}$ of $f$
in $\mathrm{Diff}^1(M)$ such that for every $g\in  \mathscr{N}$
the maximal invariant set $\Lambda_g$ of $g$ in $U$ has a
dominated splitting $ T_{\Lambda_g}=E_g \oplus F_g$ with $\dim
(E_g )= \dim (E)$ such that
\begin{enumerate}
\item
the bundles of the dominated splitting  depend continuously with
the point $x$ and the map $g$,
\item
there are continuous open cone fields $\mathcal{C}_E$ and
$\mathcal{C}_F$ defined on $U$ with $E(x) \subset
\mathcal{C}_E(x)$ and $F(x) \subset \mathcal{C}_F(x)$ such that
for every $g\in \mathscr{N}$ it holds:
\begin{itemize}
\item
$Dg^{-1} (\mathcal{C}_E(x)) \subset \mathcal{C}_E( g^{-1}(x))$ if
$x, g^{-1} (x) \in U$ and
\item
$Dg (\mathcal{C}_F(x)) \subset \mathcal{C}_F(g(x))$ if   $x, g
(x) \in U$.
\end{itemize}
\end{enumerate}
\end{prop}

\bibliographystyle{plain2}
\bibliography{BarDiaPer-submission}

\end{document}